\documentclass [twoside,reqno,  12pt] {amsart}

\usepackage{hyperref}

\usepackage{amsfonts}
\usepackage{amssymb}
\usepackage{a4}
\usepackage{color}
\usepackage{amscd,amssymb,amsfonts}

\usepackage{picture,xcolor}

\usepackage{etoolbox}

\usepackage{tikz}

\newtheorem{thm}{Theorem}[section]

\newtheorem{lem}[thm]{Lemma}
\newtheorem{prop}[thm]{Proposition}

\newtheorem{defn}[thm]{Definition}
\newtheorem{rem}[thm]{Remark}

\theoremstyle{definition}

\numberwithin{equation}{section}

\renewcommand{\Re}{\hbox{Re}\,}

\newcommand{\C}{\mathbb{C}}

\newcommand{\N}{\mathbb{N}}

\newcommand{\R}{\mathbb{R}}

\newcommand{\supp}{\operatorname{supp}}

\newcommand{\Z}{\mathbb{Z}}

\def\tilde{\widetilde}
\def \bfo {\begin {eqnarray*} }
\def \efo {\end {eqnarray*} }
\def \ba {\begin {eqnarray*} }
\def \ea {\end {eqnarray*} }
\def \beq {\begin {eqnarray}}
\def \eeq {\end {eqnarray}}
\def \supp {\hbox{supp }}

\def \dist {\hbox{dist}}

\def \p {\partial}

\def\tilde{\widetilde}
\def \bfo {\begin {eqnarray*} }
\def \efo {\end {eqnarray*} }
\def \ba {\begin {eqnarray*} }
\def \ea {\end {eqnarray*} }
\def \beq {\begin {eqnarray}}
\def \eeq {\end {eqnarray}}
\def \supp {\hbox{supp }}

\def \dist {\hbox{dist}}

\def \p {\partial}

\begin{document}

\title[Calder\'{o}n Problem for quasilinear conductivities]{The Calder\'{o}n inverse problem for isotropic quasilinear conductivities}

\author[C\^{a}rstea]{C\u{a}t\u{a}lin I. C\^{a}rstea}

\address
{School of Mathematics, Sichuan University, Chengdu, Sichuan, 610064, P.R. China} 

\email{catalin.carstea@gmail.com}

\author[Feizmohammadi]{Ali Feizmohammadi}

\address{The Fields Institute for Research in Mathematical Sciences\\
Toronto, Ontario M5T 3J1\\
Canada}

\email{afeizmoh@fields.utoronto.ca}

\author[Kian]{Yavar Kian}

\address
        {Y. Kian, Aix Marseille Univ\\ 
        Universit\'e de Toulon, CNRS\\
        CPT, Marseille, France}

\email{yavar.kian@univ-amu.fr}

\author[Krupchyk]{Katya Krupchyk}
\address
        {K. Krupchyk, Department of Mathematics\\
University of California, Irvine\\
CA 92697-3875, USA }

\email{katya.krupchyk@uci.edu}

\author[Uhlmann]{Gunther Uhlmann}

\address
       {G. Uhlmann, Department of Mathematics\\
       University of Washington\\
       Seattle, WA  98195-4350\\
       USA\\
        and Institute for Advanced Study of the Hong Kong University of Science and Technology}
\email{gunther@math.washington.edu}

\maketitle
\begin{abstract}
We prove a global uniqueness result for the Calder\'{o}n inverse problem for a general quasilinear isotropic conductivity equation on a bounded open set with smooth boundary in dimension $n\ge 3$.  Performing higher order linearizations of the nonlinear Dirichlet--to--Neumann map, we reduce the problem of the recovery of the differentials of the quasilinear conductivity, which are symmetric tensors,  to a completeness property for certain anisotropic products of solutions to the linearized equation. The completeness property is established using complex geometric optics solutions to  the linearized conductivity equation, whose amplitudes concentrate near suitable two dimensional planes.
\end{abstract}

\section{Introduction}

Let $\Omega \subset \R^n$, $n \geq 3$,  be a  bounded open set with $C^\infty$ boundary, and let us start by considering conductivity functions 
$$ \gamma: \overline{\Omega} \times \C \times \C^n\to \C,$$
that satisfy the following two assumptions:
\begin{enumerate}
\item[(H1)]{$0<\gamma(\cdot,0,0) \in C^{\infty}(\overline{\Omega})$,}
\item[(H2)]{the map $\C\times \C^n \ni (\rho,\mu) \to\gamma(\cdot,\rho,\mu)$ is holomorphic with values in  the H\"{o}lder space $C^{1,\alpha}(\overline{\Omega})$ for some $\alpha \in (0,1)$. }  
\end{enumerate} 
Given a conductivity function $\gamma$ as above, we consider the boundary value problem
\begin{equation}
\label{pf0}
\begin{aligned}
\begin{cases}
\nabla\cdot\left(\gamma(x,u,\nabla u)\nabla u\right)=0\,\quad &\text{in $\Omega$},
\\
u= f\,\quad &\text{on $\p \Omega$.}\\
\end{cases}
    \end{aligned}
\end{equation}
Arguing as in \cite[Appendix B]{KKU20},  we see that under the assumptions (H1) and (H2), there exist  $\delta>0$ and $C>0$ such that given any 
$$ f \in B_{\delta}(\p \Omega):=\{f \in C^{2,\alpha}(\p \Omega)\,:\, \|f\|_{C^{2,\alpha}(\p \Omega)}< \delta\},$$   
the problem \eqref{pf0} has a unique solution $u=u_{f} \in C^{2,\alpha}(\overline{\Omega})$ satisfying 
$\|u\|_{C^{2,\alpha}(\overline{\Omega})}< C\delta$. We define the Dirichlet-to-Neumann map associated with \eqref{pf0} via the mapping
\begin{equation}
\label{eq_int_DNmap}
 \Lambda_{\gamma}(f)= (\gamma(x,u,\nabla u)\p_\nu u)|_{\p \Omega},
\end{equation}
where $f \in B_{\delta}(\p \Omega)$, $u=u_{f}$,  and $\nu$ is the unit outer normal to $\p \Omega$. 

Our inverse problem can now be cast as follows: does the knowledge of the Dirichlet-to-Neumann map $\Lambda_\gamma$ uniquely determine a general quasilinear conductivity $\gamma$? Note that if the conductivity is assumed to be independent of $u$ and $\nabla u$, then this is the well known Calder\'{o}n problem for isotropic conductivities introduced in \cite{Ca80}, which is motivated by applications where one is interested in determining the isotropic conductivity of a medium $\Omega$ by applying voltage on the boundary $\p\Omega$ and subsequently measuring the induced current flux on $\partial\Omega$. This problem, which is also called the Electrical Impedance Tomography (EIT) problem, see \cite{U},  has many applications in different scientific branches including medical imaging by improving the early detection of breast cancer, see \cite{ZG}, as well as in seismology and geophysical exploration, see \cite{ZK}. 

In this paper we consider the EIT problem in the more general context where the unknown conductivity is not only depending on the space variable $x\in\Omega$ but that it also depends on the solution and its gradient. This corresponds to a general formulation of the Calder\'{o}n problem, where the space-dependent conductivity function is replaced by a more general quasilinear term. Beside these motivations, we recall that the recovery of a general quasilinear conductivity corresponds to an open problem whose investigation started in \cite{Su96,SuUh97}, see also \cite[Section 1.1]{MU20} for more details.

In this paper  we give an affirmative answer to the Calder\'on problem for quasilinear conductivities that satisfy (H1), (H2). Precisely, we prove the following theorem as our first main result.
\begin{thm}
\label{main_thm}
Let $\Omega \subset \R^n$, $n\geq 3$, be a bounded open set with $C^\infty$ boundary. Assume that $\gamma_1,\gamma_2: \overline{\Omega}\times \C\times \C^n\to \C$ satisfy (H1) and (H2). Suppose that there holds:
\begin{equation}
\label{eq_int_DNmap_equality}
 \Lambda_{\gamma_1}(f)=\Lambda_{\gamma_2}(f),\qquad \forall\, f \in B_{\delta}(\p \Omega).
\end{equation}
Then,
$$ \gamma_1=\gamma_2 \quad \text{in}\quad \overline{\Omega}\times \C\times \C^n.$$
\end{thm}

Let us proceed to describe the main ideas of the proof of Theorem~\ref{main_thm}. Since $\gamma_1$ and $\gamma_2$ satisfy (H2), letting $\lambda=(\rho,\mu)=(\lambda_0,\lambda_1,\dots,\lambda_n)\in \C\times \C^n$, we may write by Taylor's formula,
\begin{equation}
\label{gamma_k_form}
\gamma_j(x,\lambda) = \sum_{k=0}^{\infty} \frac{1}{k!}\gamma_j^{(k)}(x,0;\underbrace{\lambda, \dots,\lambda}_{k\text{ times}}), \quad x\in \Omega, \quad j=1,2.
\end{equation}
Here $\gamma_j^{(k)}(x,0)$ is the $k$th differential of the holomorphic function $\lambda\mapsto \gamma_j(x,\lambda)$ at $\lambda=0$, which is a symmetric tensor of rank $k$, given by 
\begin{equation}
\label{gamma_k_form_2}
\gamma_j^{(k)}(x,0;\lambda, \dots,\lambda)= \sum_{j_1,\ldots,j_k=0}^n(\p_{\lambda_{j_1}} \dots \p_{\lambda_{j_k}}\gamma_j)(x,0)\lambda_{j_1}\dots\lambda_{j_k}, \quad x\in \Omega.
\end{equation}
The power series in \eqref{gamma_k_form} converges in $C^{1,\alpha}(\overline{\Omega})$ topology.  First, performing the first order linearization of the Dirichlet problem \eqref{pf0} and the Dirichlet--to--Neumann map \eqref{eq_int_DNmap}, we obtain inverse boundary problems for the linear conductivity equations with conductivities $0<\gamma_j(x,0)\in C^\infty(\overline{\Omega})$, and relying on \cite[Theorem 0.1]{SU87}, we conclude 
$\gamma_1(\cdot,0)= \gamma_2(\cdot,0)$ in $\Omega$.  Using the $m$th order linearization of \eqref{pf0} and \eqref{eq_int_DNmap}, $m\ge 2$,  we reduce the proof of the equality of tensors of rank $m-1$, 
\[
\gamma_1^{(m-1)}(\cdot,0)=\gamma_2^{(m-1)}(\cdot,0)\quad \text{in}\quad \Omega,
\]
to the completeness property of certain anisotropic products of solutions to the linearized equation, claimed in Proposition~\ref{prop1} below. The idea of higher order linearizations was introduced in the context of wave equations in \cite{KLU18} and later adapted to elliptic equations in \cite{FO20,LLLSa20}, see also \cite{Su96,SuUh97} for a second order linearization technique.  In the statement of Proposition~\ref{prop1}, $\pi(m+1)$ stands for the set of all distinct permutations of $\{1,\ldots,m+1\}$. Also, given any two vectors $v,w \in \C^n$, the notation $v\cdot w$ stands for the bilinear extension to $\C^n$ of the Euclidean inner product on $\R^n$, i.e. $v\cdot w = \sum_{j=1}^n v_j w_j.$
\begin{prop}
\label{prop1}
Let $\Omega \subset \R^n$, $n\geq 3$, be a  bounded open set with $C^\infty$ boundary. Let $0<\gamma_0 \in C^{\infty}(\overline{\Omega})$. Let $m\in \N$ and let $T$ be a continuous function on $\overline{\Omega}$ with values in the space of symmetric tensors of rank $m$. Suppose that
\begin{equation}
\label{int_iden_density}
\begin{aligned}
\sum_{(l_1,\dots, l_{m+1})\in \pi(m+1)}\sum_{j_1,\dots, j_{m}=0}^n \int_\Omega
  T^{j_1\dots j_m}(x)
   (u_{l_1},\nabla u_{l_1})_{j_1}\dots (u_{l_{m}},\nabla u_{l_{m}})_{j_{m}} \\
   \nabla u_{l_{m+1}}\cdot \nabla u_{m+2}dx=0,
\end{aligned}
\end{equation}
for all  $u_l\in C^\infty(\overline{\Omega})$ solving $\nabla\cdot (\gamma_0 \nabla u_l)=0$ in $\Omega$, $l=1,\dots, m+2$. Then $T$ vanishes identically on $\Omega$. Here $(u_{l},\nabla u_{l})_{j}$, $j=0, 1, \dots, n$, stands for the $j$th component of the vector $(u_{l},\p_{x_1} u_{l}, \dots, \p_{x_n} u_{l})$, and in particular, $(u_{l},\nabla u_{l})_{0}=u_{l}$. 
\end{prop}

In the case when $m=1$, the proof of Proposition~\ref{prop1} basically follows from  a polarization trick and  the fact that 
\[
\text{span}\{\gamma_0 \nabla v_1\cdot \nabla v_2: v_j\in C^\infty(\overline{\Omega}), \nabla  \cdot (\gamma_0\nabla v_j)=0, j=1,2\}
\]
is dense in $L^2(\Omega)$, see \cite[Proposition 3.1]{KLassSil20}. In the case when $m\ge 2$, we observe that there are at least four solutions in the integral identity \eqref{int_iden_density}, and we shall use crucially this observation. To explain the idea, let $m=2$.  We pick any point $p\in \Omega$, and any vectors $\zeta,\tilde \zeta\in\C^n$ such that $\zeta\cdot\zeta=\tilde \zeta\cdot\tilde \zeta=0$, $\text{Re}\,\zeta=\text{\Re}\,\tilde \zeta$,  $|\text{Re}\,\zeta|=1$, and $\text{Im}\,\zeta$, $\text{Im}\,\tilde\zeta$ are linearly independent, and test \eqref{int_iden_density} against two complex geometric optics (CGO) solutions, whose amplitudes are localized near the two dimensional plane passing through the point $p$ and spanned by $\text{Re}\, \zeta$,  $\text{Im}\, \zeta$, and two more CGO solutions, whose amplitudes are localized near the two dimensional plane passing through the point $p$ and spanned by $\text{Re}\, \zeta$,  $\text{Im}\, \tilde \zeta$, see \cite{GU01} for similar ideas.  Thus, the product of amplitudes of four such solutions is localized to the ray
\[
\{x\in\R^n: x=p+t\text{Re}\,\zeta,\ t\in\R\},
\]
leading to the fact that  the Fourier transform of the function 
\[
t\mapsto \sum_{j_1,j_2=1}^n T^{j_1j_2}(p+t\text{Re}\,\zeta)\zeta_{j_1}\tilde \zeta_{j_2}
\]
along this ray vanishes. Taking $t=0$, recalling that $p\in \Omega$ is arbitrary, and making suitable choices for vectors $\zeta$ and $\tilde \zeta$, we show that the tensor $T=0$ in $\Omega$.

Let us remark that the proof of the well-posedness of the Dirichlet problem \eqref{pf0} as well as of holomorphic dependence of the solution $u_f$ on the boundary data $f \in B_{\delta}(\p \Omega)$, which is crucial for the higher order linearizations, is established by means of the implicit function theorem for holomorphic maps between complex Banach spaces, see  \cite[Appendix B]{KKU20}. Next, we would like to consider quasilinear conductivities $\gamma(x,\rho,\mu)$ which no longer depend holomorphically on $\rho$ and $\mu$. In doing so, we have proved in Theorem \ref{thm_well-posedness} the well-posedness of the Dirichlet problem \eqref{pf0} relying on the implicit function theorem for $C^\infty$ maps between real Banach spaces,  and in view of Theorem \ref{thm_well-posedness}, we first assume that  the function $\gamma: \overline{\Omega} \times \R \times \R^n\to \R$ satisfies the following conditions:
\begin{enumerate}
\item[(A1)]{$0<\gamma(\cdot,0,0) \in C^{\infty}(\overline{\Omega})$,} 
\item[(A2)] {the map $\R\times \R^n \ni (\rho,\mu) \to\gamma(\cdot,\rho,\mu)$ is $C^\infty$ with values in the H\"{o}lder space $C^{1,\alpha}(\overline{\Omega};\R)$ for some $\alpha \in (0,1)$.}  
\end{enumerate} 
Thanks to Theorem \ref{thm_well-posedness},  under the assumptions (A1) and (A2), there exist  $\delta>0$ and $C>0$ such that given any 
$$ f \in B_{\delta}(\p \Omega,\R):=\{f \in C^{2,\alpha}(\p \Omega;\R)\,:\, \|f\|_{C^{2,\alpha}(\p \Omega;\R)}< \delta\},$$   
the problem \eqref{pf0} admits a unique solution $u=u_{f} \in C^{2,\alpha}(\overline{\Omega};\R)$ satisfying 
$\|u\|_{C^{2,\alpha}(\overline{\Omega};\R)}< C\delta$. Associated to \eqref{pf0}, we define the Dirichlet-to-Neumann map $\Lambda_{\gamma}$ as in \eqref{eq_int_DNmap}.  Our second main result is as follows. 

\begin{thm}
\label{main_thm_2}
Let $\Omega \subset \R^n$, $n\geq 3$, be a bounded open set with $C^\infty$ boundary. Assume that $\gamma_1,\gamma_2: \overline{\Omega} \times \R \times \R^n\to \R$ satisfy (A1) and (A2). If 
\[
\Lambda_{\gamma_1}(f)=\Lambda_{\gamma_2}(f)\qquad \forall\, f \in B_{\delta}(\p \Omega;\R), 
\]
then for all $|\alpha|\ge 0$, we have  
\[
 \p^{\alpha}_{\rho,\mu}\gamma_1(\cdot,0,0)=\p^{\alpha}_{\rho,\mu}\gamma_2(\cdot,0,0) \quad \text{in}\quad \overline{\Omega}.
 \]
\end{thm} 
The proof of Theorem \ref{main_thm_2} follows along the same lines as the proof of Theorem \ref{main_thm}, and therefore, will be omitted. We should only remark that all the integral identities obtained in the proof will be valid for real valued solutions to the linearized conductivity equation 
\begin{equation}
\label{eq_int_100_1}
\nabla \cdot (\gamma(x,0,0)\nabla u)=0\quad \text{in}\quad \Omega.
\end{equation} 
As $\gamma(x,0,0)$ is real valued, given a complex valued solution $u$ to \eqref{eq_int_100_1}, we have  $\text{Re}\, u$ and $\text{Im}\, u$ are also solutions to \eqref{eq_int_100_1}, and therefore, all the integral identities extend to complex valued solutions, see \cite[Lemma 2.1]{LLLSb20}.

Finally, let us consider conductivities $\gamma(x,\rho, \mu)$ which are smooth in $\rho$ but real analytic in $\mu$, and in this case we shall show that we can recover the entire conductivity as in Theorem \ref{main_thm}. Specifically, let $\gamma: \overline{\Omega} \times \R \times \R^n\to \R$ satisfy the following conditions:
\begin{enumerate}
\item[(B1)]{$0<\gamma(\cdot,\cdot,0) \in C^{\infty}(\overline{\Omega}\times\R)$,} 
\item[(B2)] {the map $\R\times \R^n \ni (\rho,\mu) \to\gamma(\cdot,\rho,\mu)$ is $C^\infty$ with values in the H\"{o}lder space $C^{1,\alpha}(\overline{\Omega};\R)$ for some $\alpha \in (0,1)$. Moreover, for each $(x,\rho)\in\Omega\times \R$, the map $\R^n \ni \mu \mapsto\gamma(x,\rho,\mu)\in \R$ is real analytic.}  
\end{enumerate} 
Let $\rho\in \R$ and consider the boundary value problem
\begin{equation}
\label{pf0_smooth_rho}
\begin{aligned}
\begin{cases}
\nabla\cdot\left(\gamma(x,u,\nabla u)\nabla u\right)=0\,\quad &\text{in $\Omega$},
\\
u=\rho+ f\,\quad &\text{on $\p \Omega$.}\\
\end{cases}
    \end{aligned}
\end{equation}
It is established in Theorem \ref{thm_well-posedness} that under the assumptions (B1) and (B2), for each $\rho\in\R$,  there exist  $\delta_\rho>0$ and $C_\rho>0$ such that when
$$ f \in B_{\delta_\rho}(\p \Omega,\R):=\{f \in C^{2,\alpha}(\p \Omega;\R)\,:\, \|f\|_{C^{2,\alpha}(\p \Omega;\R)}< \delta_\rho\},$$   
the problem \eqref{pf0_smooth_rho} has a unique solution $u=u_{\rho, f} \in C^{2,\alpha}(\overline{\Omega};\R)$ satisfying 
$\|u-\rho\|_{C^{2,\alpha}(\overline{\Omega};\R)}< C_\rho\delta_\rho$. Associated to \eqref{pf0_smooth_rho}, we define the Dirichlet-to-Neumann map as follows 
\begin{equation}
\label{eq_int_DNmap_smooth_rho}
 \Lambda_{\gamma}(\rho+f)= (\gamma(x,u,\nabla u)\p_\nu u)|_{\p \Omega},
\end{equation}
where $\rho\in\R$,  $f \in B_{\delta_\rho}(\p \Omega;\R)$, and $u=u_{\rho,f}$.  Our third main result is as follows. 

\begin{thm}
\label{main_thm_3}
Let $\Omega \subset \R^n$, $n\geq 3$, be a bounded open set with $C^\infty$ boundary. Assume that $\gamma_1,\gamma_2: \overline{\Omega}\times \R\times \R^n\to \R$ satisfy (B1) and (B2). If \begin{equation}
\label{eq_int_DNmap_equality_2}
\Lambda_{\gamma_1}(\rho +f)=\Lambda_{\gamma_2}(\rho+f)\qquad \forall \rho\in\R, \quad \forall\, f \in B_{\delta}(\p \Omega;\R).
\end{equation}
Then,
$$ \gamma_1=\gamma_2 \quad \text{in}\quad \overline{\Omega}\times \R\times \R^n.$$
\end{thm} 

\begin{rem}
Comparing Theorem \ref{main_thm_3}, where there is no analyticity assumption in $\rho$ for $\gamma_j(x,\rho,\mu)$,  with Theorem \ref{main_thm}, we note that the assumption (B1) in Theorem \ref{main_thm_3} is stronger than the corresponding assumption (H1) in Theorem \ref{main_thm}, and the requirement \eqref{eq_int_DNmap_equality_2} in Theorem \ref{main_thm_3} is stronger than the corresponding requirement \eqref{eq_int_DNmap_equality} in Theorem \ref{main_thm}.
\end{rem}

The proof of Theorem~\ref{main_thm_3} follows along the same lines as the proofs of Theorem~\ref{main_thm} and Theorem~\ref{main_thm_2}. First, since $\gamma_1$ and $\gamma_2$ satisfy (B2),  by Taylor's formula, we may represent  $\gamma_j(x,\rho, \mu)$ as the sum of a convergent power series, 
\begin{equation}
\label{gamma_k_form_smooth}
\gamma_j(x,\rho, \mu) = \sum_{k=0}^{\infty} \frac{1}{k!}\gamma_j^{(k)}(x,\rho, 0;\underbrace{\mu, \dots,\mu}_{k\text{ times}}),\quad x\in \Omega,\quad \rho\in\R, \quad \mu\in\text{neigh}(0,\R^n),  
\end{equation}
$j=1,2$. Here $\gamma_j^{(k)}(x, \rho, 0)$ is the $k$th differential of the real analytic function $\mu\mapsto \gamma_j(x,\rho, \mu)$ at $\mu=0$, which is a symmetric tensor of rank $k$, given by 
\[
\gamma_j^{(k)}(x,\rho,0;\mu, \dots,\mu)= \sum_{j_1,\ldots,j_k=1}^n(\p_{\mu_{j_1}} \dots \p_{\mu_{j_k}}\gamma_j)(x,\rho, 0)\mu_{j_1}\dots\mu_{j_k}, \quad x\in \Omega, \, \rho\in\R.
\]
First, performing the first order linearization of the Dirichlet problem \eqref{pf0_smooth_rho} and the Dirichlet--to--Neumann map \eqref{eq_int_DNmap_smooth_rho}, we obtain the inverse boundary problems for the linear conductivity equations with conductivities $0<\gamma_j(x,\rho,0)\in C^\infty(\overline{\Omega})$, and relying on \cite[Theorem 0.1]{SU87} and the observation discussed after \eqref{eq_int_100_1}, we conclude 
$\gamma_1(\cdot,\cdot,0)= \gamma_2(\cdot,\cdot,0)$ in $\Omega\times \R$.  Using the $m$th order linearization of \eqref{pf0_smooth_rho} and \eqref{eq_int_DNmap_smooth_rho}, $m\ge 2$, combined with the observation above,  we reduce the proof of the equality of tensors of rank $m-1$, 
\[
\gamma_1^{(m-1)}(\cdot,\cdot,0)=\gamma_2^{(m-1)}(\cdot,\cdot,0)\quad \text{in}\quad \Omega\times \R,
\]
to the following density result, similar to Proposition \ref{prop1}. 
\begin{prop}
\label{prop1_smooth_rho}
Let $\Omega \subset \R^n$, $n\geq 3$, be a  bounded open set with $C^\infty$ boundary. Let $0<\gamma_0 \in C^{\infty}(\overline{\Omega})$. Let $m\in \N$ and let $T$ be a continuous function on $\overline{\Omega}$ with values in the space of symmetric tensors of rank $m$. Suppose that
\begin{align*}
\sum_{(l_1,\dots, l_{m+1})\in \pi(m+1)}\sum_{j_1,\dots, j_{m}=1}^n \int_\Omega
  T^{j_1\dots j_m}(x)
   \p_{x_{j_1}}u_{l_1}\dots \p_{x_{j_m}}u_{l_{m}}\nabla u_{l_{m+1}}\cdot \nabla u_{m+2}dx=0,
\end{align*}
for all  $u_l\in C^\infty(\overline{\Omega})$ solving $\nabla\cdot (\gamma_0 \nabla u_l)=0$ in $\Omega$, $l=1,\dots, m+2$. Then, $T$ vanishes identically on $\Omega$.  
\end{prop}
The proof of Proposition \ref{prop1_smooth_rho} is contained in the proof of Proposition \ref{prop1}.  Therefore, the proof of  Proposition \ref{prop1_smooth_rho} and Theorem \ref{main_thm_3} will be omitted. 

Before closing the introduction, let us review some of the previous literature on the Calder\'{o}n  problem. There have been numerous studies on the Calder\'{o}n  problem for linear conductivities with the work \cite{SU87} being one of the first major contributions. Here the authors establish uniqueness in dimensions three and higher for smooth conductivities in a linear equation and thus give an affirmative answer to the problem stated by Calder\'on in \cite{Ca80}. They also introduce the main strategy for treating this problem that has appeared in almost all of the works on the subject that is based on the construction of CGO solutions. Since the original work of \cite{SU87}, the Calder\'{o}n problem has  received a lot of attention and different extensions of the work of \cite{SU87} have been considered thus far. This includes results in dimension two \cite{B,Na96}, results with less regular conductivities \cite{AP,HT}, results with measurements restricted to some portion of the boundary \cite{BU,IUY,KSU} and results in the setting of Riemannian manifolds \cite{DKSU09,GT,LU, KLS20}.

In contrast to the above mentioned results, the recovery of nonlinear conductivities has received less attention. One of the first important contribution devoted to this problem, can be found in \cite{Su96} where the author adapts a first order linearization idea near constant functions to derive a uniqueness result. We mention that the linearization idea was introduced by \cite{I1} for parabolic equations and considered in \cite{IN95,IS94} for elliptic equations. The results of  \cite{Su96} correspond to the recovery of conductivities of the form $\gamma(x,u)$, $x\in\Omega$, $u\in\mathbb R$, depending on the space variable and the solution of the non-linear equation. This work has been extended by \cite{SuUh97} who considered more general second order nonlinear terms still depending only on the space variable and the solution of the non-linear equation. 

In \cite{HS}, the authors considered the recovery of matrix valued quasilinear terms of the form $A(x,\nabla u)$ appearing in an elliptic equation of the form $\nabla\cdot A(x,\nabla u)=0$ in dimension two. In our context, the result of \cite{HS} can be seen as the recovery of conductivities of the form $\gamma(x,\nabla u)$, $x\in\Omega$, depending on the space variable and the gradient of the solution of the nonlinear equation. More recently, the works of \cite{EPS,MU20,Sh19} have been devoted to the recovery of nonlinear conductivities independent of the space variable. The first work dealing with the recovery of conductivities the form $\gamma(x, u,\omega\cdot\nabla u)$, $x\in\Omega$, $\omega\in \R^n$, $|\omega|=1$,  can be found in \cite{KKU20}. This work, which is based on the the higher order linearization approach initiated by \cite{KLU18}, considers not only the recovery of some class of conductivities having the dependency with respect to the space variable, the solution and its derivative, but it is also stated with data restricted to some arbitrary portion of the boundary. 

Without being exhaustive, we mention also the works of \cite{Cat1,Cat2,CNV19,CFa20, CFb20} devoted to the recovery of similar type of nonlinear terms for certain classes of elliptic nonlinear equations and the works of \cite{FO20,KU19,KU20_partial,KU20,LLLSa20, LLLSb20, LLST20,  LZ20} devoted to  inverse problems for  semilinear elliptic equations.

The paper is organized as follows. In Section~\ref{section_main_thm} we show that Theorem~\ref{main_thm} follows from the completeness property of Proposition~\ref{prop1} via the method of higher order linearizations.  In Section~\ref{sec_cgo} we give a brief review of the classical complex geometric optics solutions to the linear conductivity equation, achieving good remainder estimates in $C^1(\overline{\Omega})$, and  accommodating specific choices of the amplitudes. We also present the proof of the construction of such solutions in Appendix \ref{proof_CGO} for the convenience of the reader.  In Section~\ref{section_prop1} we use complex geometric optics solutions of Section~\ref{sec_cgo} to establish the completeness property of Proposition~\ref{prop1}.  Finally, Appendix \ref{App_well_posedness} contains the well-posedness of the Dirichlet problem for our conductivity equations in the case of boundary data close to a constant one and for real valued solutions.

\section{Proof of Theorem~\ref{main_thm}. Reduction to a completeness problem via linearization}
\label{section_main_thm}

The main aim of this section is to show that Theorem~\ref{main_thm} indeed follows from the method of higher order linearization together with the completeness property of Proposition~\ref{prop1}. 

Let $\varepsilon=(\varepsilon_1,\dots, \varepsilon_m)\in\C^m$, $m\ge 1$, and let $f_1,\dots, f_m\in C^\infty(\p \Omega)$.  In view of \eqref{gamma_k_form}, \eqref{gamma_k_form_2},  the Dirichlet problem \eqref{pf0} for conductivity $\gamma_j$ with the boundary data $f=\varepsilon_1f_1+\cdot+\varepsilon_m f_m$ can be written as follows, 
\begin{equation}
\label{eq_Dirichlet_100}
\begin{cases}
\nabla\cdot(\gamma_j(x,0)\nabla u_j)\\
\quad \quad \quad \quad \quad +\nabla\cdot \bigg(\sum_{k=1}^\infty \frac{1}{k!}\sum_{j_1,\ldots,j_k=0}^n(\p_{\lambda_{j_1}} \dots \p_{\lambda_{j_k}}\gamma_j)(x,0)\\
\quad \quad \quad \quad \quad \quad \quad \quad \quad \quad \quad (u_j,\nabla u_j)_{j_1}\dots (u_j,\nabla u_j)_{j_k} \nabla u_j
\bigg) =0 & \text{in }\Omega,\\
u_j=\varepsilon_1f_1+\cdot+\varepsilon_m f_m & \text{on }\p \Omega,
\end{cases}
\end{equation}
with $j=1,2$. Here and in what follows $(u_j,\nabla u_j)_{j_l}$ stands for the $j_l$th component of the vector $(u_j,\p_{x_1}u_j, \cdots, \p_{x_n}u_j)$, and in particular, $(u_j,\nabla u_j)_{0}=u_j$.  Arguing as in \cite[Appendix B]{KKU20}, we see that for all $|\varepsilon|$ sufficiently small, the problem \eqref{eq_Dirichlet_100} has a unique small solution $u(\cdot;\varepsilon)\in C^{2,\alpha}(\overline{\Omega})$, which is holomorphic in $\varepsilon$ in a neighborhood of $\varepsilon=0$, and moreover, $\Lambda_{\gamma_j}(\varepsilon_1f_1+\cdot+\varepsilon_m f_m)$, depends holomorphically on $\varepsilon$. 

We use induction argument on $m\ge 1$ to show that the equality
\begin{equation}
\label{eq_Dirichlet_100_1}
\Lambda_{\gamma_1}(\varepsilon_1f_1+\cdot+\varepsilon_m f_m)=\Lambda_{\gamma_2}(\varepsilon_1f_1+\cdot+\varepsilon_m f_m),
\end{equation}
for all  $|\varepsilon|$ sufficiently small and all $f_1,\dots, f_m\in C^\infty(\p \Omega)$, implies that 
\begin{equation}
\label{eq_Dirichlet_100_1_1}
\gamma_1^{(m-1)}(x,0)=\gamma_1^{(m-1)}(x,0), \quad x\in \Omega.
\end{equation}
To proceed, using \eqref{gamma_k_form}, \eqref{gamma_k_form_2}, we write \eqref{eq_Dirichlet_100_1} as follows
\begin{equation}
\label{eq_Dirichlet_100_2}
\begin{aligned}
&(\gamma_1(x,0)\p_\nu u_1)|_{\p \Omega} - (\gamma_2(x,0)\p_\nu u_2)|_{\p \Omega}\\
&+\bigg( \sum_{k=1}^\infty \frac{1}{k!}\sum_{j_1,\ldots,j_k=0}^n(\p_{\lambda_{j_1}} \dots \p_{\lambda_{j_k}}\gamma_1)(x,0)(u_1,\nabla u_1)_{j_1}\dots (u_1,\nabla u_1)_{j_k} \p_\nu u_1\bigg)\bigg|_{\p \Omega}\\
&-\bigg(\sum_{k=1}^\infty \frac{1}{k!}\sum_{j_1,\ldots,j_k=0}^n(\p_{\lambda_{j_1}} \dots \p_{\lambda_{j_k}}\gamma_2)(x,0)(u_2,\nabla u_2)_{j_1}\dots (u_2,\nabla u_2)_{j_k} \p_\nu u_2 \bigg)\bigg|_{\p \Omega}=0.
\end{aligned}
\end{equation}

First let $m=1$ and consider a first order linearization of \eqref{eq_Dirichlet_100} and \eqref{eq_Dirichlet_100_2}  to show that $\gamma_1(\cdot,0)=\gamma_2(\cdot,0)$ in  $ \Omega$.  To that end, differentiating \eqref{eq_Dirichlet_100}  and \eqref{eq_Dirichlet_100_2}  with respect to $\varepsilon_1$ and evaluating at $\varepsilon_1=0$, we deduce that the function $v_j= \p_{ \varepsilon_1} u_{j}|_{\varepsilon_1=0}$ solves the problem
 \begin{equation}\label{pf_lin}
\begin{aligned}
\begin{cases}
\nabla\cdot(\gamma_j(x,0)\nabla v_j)=0\,\quad &\text{in $\Omega$},
\\
v_j=f_1\,\quad &\text{on $\p \Omega$,}\\
\end{cases}
    \end{aligned}
\end{equation}
$j=1,2$, and that 
\[
(\gamma_1(x,0)\p_\nu v_1)|_{\p \Omega}=(\gamma_2(x,0)\p_\nu v_2)|_{\p \Omega}. 
\]
Thus, the Dirichlet--to--Neumann maps for the linear conductivity equations \eqref{pf_lin} coincide, i.e.
\[
 \Lambda^{\textrm{lin}}_{\gamma_1(\cdot,0)} = \Lambda^{\textrm{lin}}_{\gamma_2(\cdot,0)},
\]
where 
\[
 \Lambda^{\textrm{lin}}_{\gamma_j(\cdot,0)}:f_1 \mapsto (\gamma_j(x,0)\p_\nu v_j)|_{\p\Omega}. 
\]
Applying \cite[Theorem 0.1]{SU87} we conclude that
\begin{equation}
\label{cond_0_eq} 
\gamma_1(\cdot,0)= \gamma_2(\cdot,0)=:\gamma_0(\cdot) \quad \text{in} \quad  \Omega,
\end{equation}
which gives us the basis of the induction.

Next, let $m\ge 2$ and for the induction hypothesis, we assume that for $k=0,1,\dots, m-2$, 
\begin{equation}
\label{ind_Gamma_hypo}
\gamma_1^{(k)}(x,0)=\gamma_1^{(k)}(x,0), \quad x\in \Omega.
\end{equation}
We shall prove that \eqref{ind_Gamma_hypo} holds for $k=m-1$. To this end, we use the method of higher order linearization as \cite{KLU18,FO20,LLLSa20}. First, as above, applying the operator $\p_{\varepsilon_l}|_{\varepsilon=0}$, $l=1,\dots,m$, to 
\eqref{eq_Dirichlet_100}, we get
 \[
\begin{cases}
\nabla\cdot(\gamma_0\nabla v^{(l)}_j)=0\,\quad &\text{in $\Omega$},
\\
v_j^{(l)}=f_l\,\quad &\text{on $\p \Omega$,}\\
\end{cases}
\]
where $v_j^{(l)}= \p_{ \varepsilon_l} u_{j}|_{\varepsilon=0}$. It follows that $v^{(l)}:=v_1^{(l)}=v_2^{(l)}\in C^\infty(\overline{\Omega})$ by the uniqueness and elliptic regularity.

Let $\beta=(\beta_1,\ldots,\beta_m)\in \{0,1,\ldots\}^{m}$ be a multi-index with $|\beta|=\sum_{j=1}^m\beta_j$. By applying the differential operator $\p_\varepsilon^\beta$ to \eqref{eq_Dirichlet_100}, first when $|\beta|=2$, and repeatedly up to $|\beta|=m-1$, and  using the induction hypothesis \eqref{ind_Gamma_hypo}, 
we deduce that 
\begin{equation}
\label{eq_Dirichlet_100_3}
\p^\beta_\varepsilon u_1|_{\varepsilon=0} = \p^\beta_\varepsilon u_2|_{\varepsilon=0},
\end{equation}
for all multi-indices $\beta$ with $|\beta|\leq m-1$. Next, let us define for each $j=1,2,$ the function $w_j= \p^{m}_{\varepsilon_1\ldots\varepsilon_m} u_{j}|_{\varepsilon=0}$.
Applying the operator $\p^{m}_{\varepsilon_1\ldots\varepsilon_m}|_{\varepsilon=0}$ to \eqref{eq_Dirichlet_100}, we see that $w_j$ satisfies
\begin{equation}
\label{two_equations}
\begin{cases}
\nabla\cdot (\gamma_0\nabla w_j)+\nabla\cdot \bigg( \frac{1}{(m-1)!}\sum_{(l_1,\dots, l_m)\in \pi(m)}\sum_{j_1,\dots, j_{m-1}=0}^n \\
 \quad \quad  (\p_{\lambda_{j_1}} \dots \p_{\lambda_{j_{m-1}}}\gamma_j)(x,0) (v^{(l_1)},\nabla v^{(l_1)})_{j_1}\dots (v^{(l_{m-1})},\nabla v^{(l_{m-1})})_{j_{m-1}}\\
\quad \quad \quad \quad  \quad \quad  \quad \quad  \nabla v^{(l_m)} \bigg) =H_{m} & \text{in }\Omega,\\
w_j=0 & \text{on }\p \Omega. 
\end{cases}
\end{equation}
Here
\begin{align*}
H_{m}= -\nabla \cdot \bigg(\p^m_{\varepsilon_1 \ldots\varepsilon_m} \bigg(
\sum_{k=1}^{m-2} \frac{1}{k!}\sum_{j_1,\ldots,j_k=0}^n(\p_{\lambda_{j_1}} \dots \p_{\lambda_{j_k}}\gamma_j)(x,0)\\
\quad \quad \quad \quad \quad \quad \quad \quad \quad \quad \quad (u_j,\nabla u_j)_{j_1}\dots (u_j,\nabla u_j)_{j_k} \nabla u_j
\bigg)\bigg|_{\varepsilon=0} \bigg) 
\end{align*}
is independent of $j=1,2$ in view of \eqref{ind_Gamma_hypo} and \eqref{eq_Dirichlet_100_3}. 

Next, letting  
$ w=w_1-w_2$ and subtracting the two equations given by \eqref{two_equations} for $j=1,2$, we deduce that
\begin{equation}
\label{final_equation_w}
\begin{cases}
\nabla\cdot (\gamma_0\nabla w)+\nabla\cdot \bigg( \frac{1}{(m-1)!}\sum_{(l_1,\dots, l_m)\in \pi(m)}\sum_{j_1,\dots, j_{m-1}=0}^n \\
 \quad \quad \big( (\p_{\lambda_{j_1}} \dots \p_{\lambda_{j_{m-1}}}\gamma_1)(x,0)- (\p_{\lambda_{j_1}} \dots \p_{\lambda_{j_{m-1}}}\gamma_2)(x,0) \big)\\
 \quad \quad  \quad \quad  \quad \quad   (v^{(l_1)},\nabla v^{(l_1)})_{j_1}\dots (v^{(l_{m-1})},\nabla v^{(l_{m-1})})_{j_{m-1}} \nabla v^{(l_m)} \bigg) =0 & \text{in }\Omega,\\
w=0 & \text{on }\p \Omega. 
\end{cases}
\end{equation}
Applying the operator $\p^{m}_{\varepsilon_1\ldots\varepsilon_m}|_{\varepsilon=0}$ to \eqref{eq_Dirichlet_100_2}, and using \eqref{cond_0_eq},  \eqref{ind_Gamma_hypo}, and \eqref{eq_Dirichlet_100_3}, we get 
\begin{equation}
\label{final_equation_w_DN}
\begin{aligned}
\gamma_0(\p_\nu w_1-\p_\nu w_2)|_{\p \Omega}+&\bigg(\frac{1}{(m-1)!}\sum_{(l_1,\dots, l_m)\in \pi(m)}\sum_{j_1,\dots, j_{m-1}=0}^n \\
&  \big( (\p_{\lambda_{j_1}} \dots \p_{\lambda_{j_{m-1}}}\gamma_1)(x,0)- (\p_{\lambda_{j_1}} \dots \p_{\lambda_{j_{m-1}}}\gamma_2)(x,0) \big)\\
  &(v^{(l_1)},\nabla v^{(l_1)})_{j_1}\dots (v^{(l_{m-1})},\nabla v^{(l_{m-1})})_{j_{m-1}} \p_\nu v^{(l_m)}\bigg)\bigg|_{\p\Omega}=0. 
\end{aligned}
\end{equation}
Let $v^{(m+1)}\in C^\infty(\overline{\Omega})$ be such that $\nabla\cdot (\gamma_0 \nabla v^{(m+1)})=0$ in $\Omega$. Multiplying  \eqref{final_equation_w} by $v^{(m+1)}$,  integrating by parts, and using \eqref{final_equation_w_DN}, we get 
\begin{equation}
\label{final_equation_w_integral}
\begin{aligned}
\sum_{(l_1,\dots, l_m)\in \pi(m)}\sum_{j_1,\dots, j_{m-1}=0}^n \int_\Omega
  \big( (\p_{\lambda_{j_1}} \dots \p_{\lambda_{j_{m-1}}}\gamma_1)(x,0)- (\p_{\lambda_{j_1}} \dots \p_{\lambda_{j_{m-1}}}\gamma_2)(x,0) \big) \\
   (v^{(l_1)},\nabla v^{(l_1)})_{j_1}\dots (v^{(l_{m-1})},\nabla v^{(l_{m-1})})_{j_{m-1}} \nabla v^{(l_m)}\cdot \nabla v^{(m+1)}dx=0,
\end{aligned}
\end{equation}
which is valid for all  $v^{(l)}\in C^\infty(\overline{\Omega})$ solving $\nabla\cdot (\gamma_0 \nabla v^{(l)})=0$ in $\Omega$, $l=1,\dots, m+1$. 
By applying Proposition~\ref{prop1} with 
\[
T^{j_1\dots j_{m-1}}(x):=(\p_{\lambda_{j_1}} \dots \p_{\lambda_{j_{m-1}}}\gamma_1)(x,0)- (\p_{\lambda_{j_1}} \dots \p_{\lambda_{j_{m-1}}}\gamma_2)(x,0),
\]
 we conclude that \eqref{eq_Dirichlet_100_1_1} holds. This completes the proof of Theorem~\ref{main_thm}.

\section{Complex geometric optics solutions to the linear conductivity equation}
\label{sec_cgo}

Let $\Omega\subset \R^n$, $n\ge 3$, be a bounded open set with $C^\infty$ boundary, and let  $0<\gamma_0\in C^\infty(\overline{\Omega})$. Consider the linear conductivity equation
\begin{equation}
\label{linear_eq}
\nabla\cdot (\gamma_0\nabla u)=0\quad \text{in}\quad \Omega. 
\end{equation}
The underlying idea in proving Proposition~\ref{prop1} is to construct a rich enough family of special solutions to \eqref{linear_eq} so that the integral identity \eqref{int_iden_density}, tested against these solutions,  forces the tensor $T$ in the statement of the proposition to be zero. For the special solutions, we use the classical complex geometric optics (CGO) solutions, with some special choices of amplitudes, enjoying some concentration properties near two dimensional planes, in the spirit of \cite{GU01}. Let us remark that CGO solutions were introduced in \cite{Ca80} for the case $\gamma_0=1$ and developed further for arbitrary positive $\gamma_0$ in \cite{SU87},  see \cite{DKSU09} for generalization of CGO solutions in the context of certain classes of Riemannian manifolds. 

As we have to work with products of more than four solutions and their first order derivatives, we need to use CGO solutions with good estimates for the remainder terms in $C^{1}(\overline{\Omega})$.  The construction of such solutions  to the conductivity equation \eqref{linear_eq} is known, and is stated in the following lemma. We give a very simple proof of it in  Appendix \ref{proof_CGO}  for the convenience of the reader and to be able to accommodate our choice of the amplitudes, see also \cite{BU10}, \cite[Proposition 2]{FO20},  \cite{LLLSa20}, \cite{KU20}. 

\begin{lem}
\label{lem_cgo}
Let  $\Omega \subset \R^n$, $n\ge 3$,  be a bounded open set with $C^\infty$ boundary and suppose that $0<\gamma_0 \in C^{\infty}(\overline{\Omega})$.  Let $\tilde \Omega\subset \R^n$ be a bounded open set with $C^\infty$ boundary such that $\Omega\subset \subset \tilde \Omega$ and let us extend $ \gamma_0$ to $\tilde \Omega\setminus\Omega$ so that the extension still denoted by $\gamma_0\in C^\infty(\overline{\tilde \Omega})$.  Let  $0\ne \zeta\in \C^n$ be such that $\zeta\cdot\zeta=0$, and $a\in C^\infty(\overline{\tilde \Omega})$ satisfy the transport equation 
\begin{equation}
\label{eq_500_9}
\zeta\cdot \nabla a=0\quad \text{in} \quad \tilde \Omega. 
\end{equation}
Then for $\lambda>0$ large enough, the conductivity equation \eqref{linear_eq} 
has solutions $U_{\lambda\zeta}\in C^\infty(\overline{\Omega})$ of the form
\begin{equation}
\label{eq_500_10}
U_{\lambda\zeta}(x)=e^{\lambda\zeta\cdot x}\gamma_0^{-1/2}(a(x)+r_{\lambda\zeta}(x))
\end{equation}
where 
\begin{equation}
\label{eq_500_11}
\|r_{\lambda\zeta}\|_{C^1(\overline{\Omega})}=\mathcal{O}(\lambda^{-1}), \quad \lambda\to \infty.
\end{equation}
\end{lem}

Let $\zeta\in \C^n$ be such that $\zeta\cdot\zeta=0$ and $|\text{\Re}\, \zeta|=|\text{Im}\,\zeta|=1$. 
When proving Proposition~\ref{prop1}, we shall work with CGO solutions \eqref{eq_500_10} with the amplitudes $a$ which depend on a given point $p\in \Omega$ and which concentrate near the two dimensional plane, passing through $p$ and spanned by $\text{\Re}\, \zeta$ and $\text{Im}\, \zeta$,  in the spirit of  \cite{GU01}. To construct such amplitudes, let us fix $\sigma \in \R$, $\delta \in (0,1)$ and let $\{\omega_j\}_{j=1}^{n-2}$ be an orthonormal set satisfying
\[
 \omega_j \cdot \textrm{Re}\,\zeta= \omega_j \cdot \textrm{Im}\,\zeta=0,\quad \text{for $j=1,2,\ldots,n-2$.}
\]
Let $\chi\in C^{\infty}_0(\R;[0,1])$ be chosen so that $\chi(t)=1$ for $|t|\leq \frac{1}{2}$ and $\chi(t)=0$ for $|t|\geq 1$. We define $a$ via
\begin{equation}
\label{h_form}
a(x)=e^{i\sigma \zeta\cdot (x-p)}\,\prod_{j=1}^{n-2}\chi\bigg(\frac{\omega_j\cdot (x-p)}{\delta}\bigg).
\end{equation}
We have that $a\in C^\infty(\R^n)$ and $\zeta\cdot\nabla a=0$ in $\R^n$. 

Let us finally remark that  $\supp(a)$ is contained in a $\delta$-neighborhood of the two plane, passing through the point $p$, spanned by the vectors $\text{Re}\,\zeta$ and $\text{Im}\,\zeta$. Indeed, letting $\Pi_\zeta=\text{span}\{\text{Re}\,\zeta, \text{Im}\,\zeta\}$ be the plane, passing through the origin, spanned by $\text{Re}\,\zeta$ and $\text{Im}\,\zeta$, and letting  $x\in \supp(a)$, we get  
\begin{equation}
\label{eq_500_12}
\dist(x-p, \Pi_\zeta)=\sqrt{\sum_{j=1}^{n-2}(\omega_j\cdot(x-p))^2}\le \sqrt{n}\delta,
\end{equation}
showing the claim.

\section{Proof of Proposition~\ref{prop1}}
\label{section_prop1}

\subsection{Proof of Proposition~\ref{prop1} in the case $m=1$}  In this case the integral identity \eqref{int_iden_density} has the form
\begin{equation}
\label{int_iden_density_m=1}
\begin{aligned}
&0=\sum_{(l_1,l_2)\in \pi(2)}\sum_{j=0}^n \int_\Omega T^j(x) (u_{l_1},\nabla u_{l_1})_{j}\nabla u_{l_2}\cdot \nabla u_3 dx\\
&=\sum_{(l_1,l_2)\in \pi(2)}\int_\Omega T^0(x) u_{l_1}\nabla u_{l_2}\cdot\nabla u_3 dx
+\sum_{(l_1,l_2)\in \pi(2)}\sum_{j=1}^n \int_\Omega T^j(x) \p_{x_j}u_{l_1}\nabla u_{l_2}\cdot \nabla u_3 dx,
\end{aligned}
\end{equation}
which holds for all $u_l\in C^\infty(\overline{\Omega})$, $l=1,2,3$, solving 
\begin{equation}
\label{lin_rand}
\nabla \cdot(\gamma_0 \nabla u)=0\quad \text{in} \quad \Omega.
\end{equation} 
First letting $u_2=1$ into \eqref{int_iden_density_m=1}, we get 
\begin{equation}
\label{int_iden_density_m=1_T_0}
\int_\Omega T^0(x) \nabla u_{1}\cdot\nabla u_3 dx=0,
\end{equation}
for all $u_1,u_3\in C^\infty(\overline{\Omega})$ solving \eqref{lin_rand}. Using the fact that 
\begin{equation}
\label{eq_density_gradients_900}
\text{span}\{\gamma_0 \nabla v_1\cdot \nabla v_2: v_j\in C^\infty(\overline{\Omega}), \nabla  \cdot (\gamma_0\nabla v_j)=0, j=1,2\}
\end{equation}
is dense in $L^2(\Omega)$, see \cite[Proposition 3.1]{KLassSil20}, we obtain from \eqref{int_iden_density_m=1_T_0} that 
\begin{equation}
\label{int_iden_density_m=1_T_0_1}
T^0=0\quad \text{in} \quad \Omega. 
\end{equation}
Hence, in view of \eqref{int_iden_density_m=1_T_0_1}, the identity  \eqref{int_iden_density_m=1} becomes
\begin{equation}
\label{int_iden_density_m=1_gen}
\sum_{(l_1,l_2)\in \pi(2)}\sum_{j=1}^n \int_\Omega T^j(x) \p_{x_j}u_{l_1}\nabla u_{l_2}\cdot \nabla u_3 dx=0,
\end{equation}
for all $u_l\in C^\infty(\overline{\Omega})$, $l=1,2,3$, solving \eqref{lin_rand}.
Setting $u_1=u_2$ in \eqref{int_iden_density_m=1_gen}, we deduce that
\begin{equation}
\label{polarized_rand_0}
 \int_{\Omega} \sum_{j=1}^nT^{j}\p_{x_j} u_1\nabla u_1\cdot \nabla u_3\,dx=0. 
\end{equation}
Let $v, w\in C^\infty(\overline{\Omega})$ be solutions to \eqref{lin_rand}, and let us choose $u_1=v+w$ and $u_3=v$. It follows from \eqref{polarized_rand_0} that
\begin{equation}
\label{polarized_rand_0_200}
\begin{aligned}
0=\int_\Omega \sum_{j=1}^nT^{j}\p_{x_j} v\nabla v\cdot \nabla v\,dx+ \int_\Omega  \sum_{j=1}^nT^{j}\p_{x_j} v \nabla w\cdot \nabla v\,dx\\
+\int_\Omega  \sum_{j=1}^nT^{j}\p_{x_j} w\nabla v\cdot \nabla v\,dx+\int_\Omega\sum_{j=1}^nT^{j}\p_{x_j} w\nabla w\cdot \nabla v\,dx.
\end{aligned}
\end{equation}
The first, second, and fourth terms in \eqref{polarized_rand_0_200} must vanish by \eqref{polarized_rand_0}, and consequently, we get 
\begin{equation}
\label{polarized_rand_0_200_1}
\int_\Omega \sum_{j=1}^nT^{j}\p_{x_j} w \nabla v\cdot \nabla v\,dx=0,
\end{equation}
for all $v, w\in C^\infty(\overline{\Omega})$ solving \eqref{lin_rand}. Finally, by polarization of \eqref{polarized_rand_0_200_1}, we obtain that
\begin{equation}
\label{polarized_eq} 
\int_\Omega  \sum_{j=1}^nT^{j}\p_{x_j} w\nabla v_1\cdot \nabla v_2\,dx=0,
\end{equation}
for all $v_1,v_2, w\in C^\infty(\overline{\Omega})$ solving \eqref{lin_rand}. Using the fact that \eqref{eq_density_gradients_900} is dense in $L^2(\Omega)$, see \cite[Proposition 3.1]{KLassSil20}, we conclude from \eqref{polarized_eq} that 
\begin{equation}
\label{polarized_rand_0_200_2}
 \sum_{j=1}^nT^{j}\p_{x_j} w=0\quad \text{in}\quad \Omega,
\end{equation}
for all $w\in C^\infty(\overline{\Omega})$ solving \eqref{lin_rand}.

We want to use CGO solutions to show that \eqref{polarized_rand_0_200_2} implies that $(T^1, \dots, T^n)$ is identically zero. To that end, letting $\zeta\in \C^n$ be  such that $\zeta\cdot\zeta =0$ and $|\text{Re}\, \zeta|=|\text{Im}\, \zeta|=1$, by Lemma \ref{lem_cgo} for all $\lambda>0$ large enough, the conductivity equation \eqref{lin_rand} has solutions of the form
\begin{equation}
\label{polarized_rand_0_200_3}
U_{\lambda\zeta}=e^{\lambda\zeta\cdot x}\gamma_0^{-1/2}(1+r_{\lambda\zeta})\in C^\infty(\overline{\Omega}),
\end{equation}
with $\|r_{\lambda\zeta}\|_{C^1(\overline{\Omega})}=\mathcal{O}(\lambda^{-1})$, as $\lambda\to \infty$. Substituting $U_{\lambda\zeta}$ given by \eqref{polarized_rand_0_200_3} into \eqref{polarized_rand_0_200_2}, and multiplying by $\lambda^{-1}$, we see that 
\begin{equation}
\label{m1_final}
\sum_{j=1}^n T^j \, \zeta_j=0\quad \text{in}\quad \Omega.
\end{equation}
Let $\xi \in \R^n$ be an arbitrary unit vector  and choose $\eta \in \R^n$ such that $|\eta|=1$ and $\xi\cdot \eta=0$. Thus,  the vector 
\[
 \zeta= \xi + i\eta 
\]
satisfies $\zeta\cdot\zeta=\overline{\zeta}\cdot\overline{\zeta}=0$. Hence, in view of \eqref{m1_final}, we conclude that
\begin{equation}
\label{polarized_rand_0_200_4}
 2\sum_{j=1}^n T^j \xi_j = \sum_{j=1}^n T^j (\zeta_j + \overline{\zeta}_j)=0 \quad \text{in}\quad \Omega.
\end{equation}
It follows from \eqref{int_iden_density_m=1_T_0_1},  \eqref{polarized_rand_0_200_4} that $T=(T^0, T^1,\dots, T^n)=0$ in $\Omega$. This completes the proof of  Proposition~\ref{prop1} in the case $m= 1$.

\subsection{Proof of Proposition~\ref{prop1} in the case $m= 2$}

When $m=2$, the integral identity \eqref{int_iden_density} has the form
\begin{equation}
\label{int_iden_density_m=2}
\begin{aligned}
0=&\sum_{(l_1,l_2,l_3)\in \pi(3)}\sum_{j,k=0}^n \int_\Omega T^{jk}(x) (u_{l_1},\nabla u_{l_1})_{j}(u_{l_2},\nabla u_{l_2})_{k}\nabla u_{l_3}\cdot \nabla u_4 dx\\
&=\sum_{(l_1,l_2,l_3)\in \pi(3)} \int_{\Omega}T^{00}(x) u_{l_1}u_{l_2}\nabla u_{l_3}\cdot\nabla u_4dx\\
&+ 
2\sum_{(l_1,l_2,l_3)\in \pi(3)} \sum_{j=1}^n\int_{\Omega}T^{0j}(x) u_{l_1} \p_{x_j}u_{l_2}\nabla u_{l_3}\cdot\nabla u_4dx\\
& +\sum_{(l_1,l_2,l_3)\in \pi(3)}\sum_{j,k=1}^n \int_\Omega T^{jk}(x) \p_{x_j}u_{l_1}\p_{x_k}u_{l_2}\nabla u_{l_3}\cdot \nabla u_4 dx,
\end{aligned}
\end{equation}
which holds for all $u_l\in C^\infty(\overline{\Omega})$, $l=1,\dots,4$ solving \eqref{lin_rand}.  

First letting $u_3=u_2=1$ in \eqref{int_iden_density_m=2}, we get 
\begin{equation}
\label{int_iden_density_m=2_0}
\int_{\Omega} T^{00}(x)\nabla u_1\cdot\nabla u_4 dx=0,
\end{equation}
 for all $u_1, u_4\in C^\infty(\overline{\Omega})$ solving \eqref{lin_rand}.  Note that the identity \eqref{int_iden_density_m=2_0} is the same as \eqref{int_iden_density_m=1_T_0}, and arguing as above, we conclude that 
\begin{equation}
\label{int_iden_density_m=2_0}
T^{00}=0\quad \text{in}\quad \Omega.
\end{equation}

Next using \eqref{int_iden_density_m=2_0} and letting $u_3=1$ in \eqref{int_iden_density_m=2}, we obtain that 
\begin{equation}
\label{int_iden_density_m=2_T_0_j}
\sum_{(l_1,l_2)\in \pi(2)}\int_{\Omega} \sum_{j=1}^n T^{0j}(x)  \p_{x_j}u_{l_1}\nabla u_{l_2}\cdot\nabla u_4dx=0,
\end{equation}
for all $u_1,u_2,u_4\in C^\infty(\overline{\Omega})$ solving \eqref{lin_rand}.  Note that the identity \eqref{int_iden_density_m=2_T_0_j} is the same as \eqref{int_iden_density_m=1_gen}, and therefore, arguing as above, we get 
\begin{equation}
\label{int_iden_density_m=2_0_T_0_j}
(T^{01},\dots, T^{0n})=0\quad \text{in}\quad \Omega.
\end{equation} 

In view of \eqref{int_iden_density_m=2_0} and \eqref{int_iden_density_m=2_0_T_0_j}, the idenity \eqref{int_iden_density_m=2} becomes
\begin{equation}
\label{int_iden_density_m=2_main_part}
\sum_{(l_1,l_2,l_3)\in \pi(3)}\sum_{j,k=1}^n \int_\Omega T^{jk}(x) \p_{x_j}u_{l_1}\p_{x_k}u_{l_2}\nabla u_{l_3}\cdot \nabla u_4 dx=0,
\end{equation}
 for all $u_l\in C^\infty(\overline{\Omega})$, $l=1,\dots,4$ solving \eqref{lin_rand}.   Here we do not have a straightforward analogue of the identity \eqref{polarized_eq} that we obtained via polarization in the previous section. On the other hand,  identity \eqref{int_iden_density_m=2_main_part} contains four solutions to the linear conductivity equation and we can use a different approach by using CGO solutions corresponding to different vectors to obtain pointwise information about the tensor $T$. We start with a definition. 

\begin{defn}[Admissible pairs of vectors]
\label{zeta_pair}
We define $\mathcal A$ as the set of all pairs of vectors $(\zeta,\widetilde{\zeta}) \in \C^n\times \C^n$ that satisfy the following properties:
\begin{itemize}
\item[(i)]{$\zeta\cdot \zeta=\widetilde{\zeta}\cdot\widetilde{\zeta}=0$,}
\item[(ii)]{$\textrm{Re}\,\zeta=\textrm{Re}\,\widetilde \zeta$,}
\item[(iii)]{$\textrm{Im}\,\widetilde{\zeta} \notin \{t\,\textrm{Im}\,\zeta\,:\,t \in \R\}.$}
\end{itemize}
\end{defn} 

To introduce suitable CGO solutions, we let $p \in \Omega$,  $(\zeta,\widetilde \zeta) \in \mathcal A$, $|\text{Re}\,\zeta|=|\text{Re}\,\tilde \zeta|=1$,  $\sigma \in \R$, $\delta \in (0,1)$, and define orthonormal sets $\{\omega_j\}_{j=1}^{n-2}$ and $\{\widetilde{\omega}_j\}_{j=1}^{n-2}$ that satisfy
\begin{equation}
\label{eq_600_-10}
 \omega_j \cdot \textrm{Re}\,\zeta= \omega_j \cdot \textrm{Im}\,\zeta=0,\quad \text{for $j=1,2,\ldots,n-2$.}
\end{equation}
and
\begin{equation}
\label{eq_600_-10_1}
\widetilde{\omega}_j \cdot \textrm{Re}\,\widetilde\zeta= \widetilde{\omega}_j \cdot \textrm{Im}\,\widetilde\zeta=0,\quad\text{for $j=1,2,\ldots,n-2$.}
\end{equation}
Let $\chi\in C^{\infty}_0(\R;[0,1])$ be such that $\chi(t)=1$ for $|t|\leq \frac{1}{2}$ and $\chi(t)=0$ for $|t|\geq 1$. We set 
\begin{equation}
\label{h_form}
a(x)=e^{i\sigma \zeta\cdot (x-p)}\,\prod_{j=1}^{n-2}\chi\bigg(\frac{\omega_j\cdot (x-p)}{\delta}\bigg), 
\end{equation}
and
\begin{equation}
\label{h_tilde_form}
\widetilde a(x) =e^{i\sigma \widetilde\zeta\cdot (x-p)}\prod_{j=1}^{n-2}\chi\bigg(\frac{\widetilde{\omega}_j\cdot (x-p)}{\delta}\bigg).
\end{equation}
We have $a, \widetilde a\in C^\infty(\R^n)$, and 
$$ \zeta\cdot \nabla a = \widetilde{\zeta} \cdot \nabla \widetilde{a}=0\quad \text{in}\quad  \R^n.$$ 

By Lemma \ref{lem_cgo}, for all $\lambda>0$ large enough, the conductivity equation \eqref{lin_rand} has solutions $U_{\lambda\zeta}, U_{-\lambda\zeta},  U_{\lambda\tilde \zeta}, U_{-\lambda\tilde \zeta} \in C^\infty(\overline{\Omega})$ of the form
\begin{equation}
\label{eq_600_1}
\begin{aligned}
&U_{\pm \lambda\zeta}(x)= e^{\pm \lambda \zeta\cdot x}\gamma_0(x)^{-\frac{1}{2}}\left( a(x)+r_{\pm \lambda\zeta}(x)\right),\\
&U_{\pm \lambda\tilde \zeta}(x)= e^{\pm \lambda \tilde \zeta\cdot x}\gamma_0(x)^{-\frac{1}{2}}\left( \tilde a(x)+r_{\pm \lambda\tilde \zeta}(x)\right),
\end{aligned}
\end{equation}
where 
\begin{equation}
\label{eq_600_2}
\|r_{\pm\lambda\zeta}\|_{C^1(\overline{\Omega})}=\mathcal{O}(\lambda^{-1}), \quad \|r_{\pm\lambda\tilde \zeta}\|_{C^1(\overline{\Omega})}=\mathcal{O}(\lambda^{-1}),
\end{equation}
as $\lambda\to \infty$. 
Note that since $(\zeta,\tilde \zeta)\in \mathcal{A}$ and $|\text{Re}\,\zeta|=|\text{Re}\,\tilde \zeta|=1$, we have 
\begin{equation}
\label{eq_600_3}
 \zeta\cdot\tilde \zeta=1-\text{Im}\zeta\cdot\text{Im}\tilde\zeta>0.
\end{equation}
Let
\begin{equation}
\label{eq_600_4}
u_1=U_{\lambda\zeta}, \quad u_2= U_{-\lambda\zeta},\quad u_{3}=\widetilde U_{\lambda\tilde \zeta},\quad \quad u_{4}=\widetilde U_{-\lambda\tilde \zeta}.
\end{equation}
Substituting \eqref{eq_600_4}, \eqref{eq_600_1},  into the integral identity \eqref{int_iden_density_m=2_main_part}, multiplying by $\lambda^{-4}$, $\lambda\to \infty$,  and using \eqref{eq_600_2}, \eqref{eq_600_3},  we get  
\[
(\zeta\cdot\tilde \zeta)\int_\Omega \sum_{j,k=1}^n T^{jk}\zeta_j\tilde \zeta_k\gamma_0^{-2}(a\widetilde a)^2dx=0,
\]
and therefore,
\begin{equation}
\label{eq_600_5}
\int_\Omega \sum_{j,k=1}^n T^{jk}\zeta_j\tilde \zeta_k\gamma_0^{-2}Fdx=0.
\end{equation}
Here 
\[
F(x)=(a(x)\widetilde a(x))^2=e^{i 2\sigma(\zeta+\tilde \zeta)\cdot (x-p)}\prod_{j=1}^{n-2}\chi\bigg(\frac{\omega_j\cdot (x-p)}{\delta}\bigg)^2 \chi\bigg(\frac{\widetilde{\omega}_j\cdot (x-p)}{\delta}\bigg)^2, 
\]
$\delta\in (0,1)$ and $\sigma\in \R$. 

We claim that in view of the fact that $(\zeta,\tilde \zeta)\in \mathcal{A}$, we have $\supp(F)$ is contained in  a $\delta$--neighborhood of the ray
\begin{equation}
\label{eq_600_6}
\{x\in \R^n: x=p+t\text{Re}\, \zeta, \ t\in \R\}.
\end{equation}
Indeed, letting $x$ be in $\supp(F)$ and letting 
\[
\Pi_\zeta=\text{span}\{\text{Re}\,\zeta, \text{Im}\,\zeta\}, \quad \Pi_{\tilde \zeta}=\text{span}\{\text{Re}\,\zeta, \text{Im}\,\tilde \zeta\},
\]
 be the two dimensional planes, passing through the origin, spanned by $\text{Re}\,\zeta$,  $\text{Im}\,\zeta$, and $\text{Re}\,\zeta$,  $\text{Im}\,\tilde \zeta$, respectively,  we have in view of  \eqref{eq_500_12}, 
\begin{equation}
\label{eq_600_7}
\dist(x-p, \Pi_\zeta)\le \sqrt{n}\delta, \quad \dist(x-p, \Pi_{\tilde\zeta})\le \sqrt{n}\delta.
\end{equation}
Using  \eqref{eq_600_7} and  
\begin{equation}
\label{eq_600_7_9000_1}
\dist(x-p, \Pi_\zeta)=|x-p-((x-p)\cdot \text{Re}\,\zeta)\text{Re}\,\zeta- ((x-p)\cdot \text{Im}\,\zeta)\text{Im}\,\zeta|,
\end{equation}
and the corresponding expression for $\dist(x-p, \Pi_{\tilde \zeta})$, 
we get 
\begin{equation}
\label{eq_600_7_9000_2}
|((x-p)\cdot \text{Im}\,\zeta)\text{Im}\,\zeta-((x-p)\cdot \text{Im}\,\tilde \zeta)\text{Im}\,\tilde \zeta|\le 2\sqrt{n}\delta.
\end{equation}
Since the vectors $\text{Im}\,\zeta$ and $\text{Im}\,\tilde\zeta$ are linearly independent, we see from \eqref{eq_600_7_9000_2} that 
\[
|((x-p)\cdot \text{Im}\,\zeta)\text{Im}\,\zeta |\le C\delta, \quad |((x-p)\cdot \text{Im}\,\tilde \zeta)\text{Im}\,\tilde \zeta |\le C\delta, 
\]
where $C>0$ depends on the angle between $\text{Im}\,\zeta$ and $\text{Im}\, \tilde \zeta$,  and $n$. 
This together with \eqref{eq_600_7}, \eqref{eq_600_7_9000_1}, gives that 
\[
|x-p -((x-p)\cdot \text{Re}\,\zeta)\text{Re}\,\zeta|\le (C+\sqrt{n})\delta,
\]
 showing the claim. 

Multiplying \eqref{eq_600_5} by $\delta^{-(n-1)}$ and taking the limit as $\delta \to 0$, we deduce that
\[
\int_\R \sum_{j,k=1}^n e^{i 4\sigma t}\gamma_0^{-2}(p+t\text{Re}\,\zeta)\,T^{j k}(p+t\text{Re}\, \zeta) \zeta_{j}\,\tilde \zeta_{k}\,dt=0,
\]
where we have extended $T$ and $\gamma_0$ to all of $\R^n$ by setting them to be zero outside $\Omega$. Since the latter expression holds for all $\sigma \in \R$ we can use inverse Fourier transform in $\sigma$ to obtain that the integrand above should vanish for all $t$ and in particular at $t=0$. Finally, since $p \in \Omega$ is arbitrary, we conclude that 
\begin{equation}
\label{eq_600_8}
\sum_{j,k=1}^n T^{j k}\zeta_{j}\,\tilde \zeta_{k}=0\quad \text{in}\quad \Omega,
\end{equation}
for all $(\zeta,\tilde \zeta)\in \mathcal{A}$, $|\text{Re}\,\zeta|=|\text{Re}\,\tilde \zeta|=1$.  

Now let $\xi\in \R^n$ be arbitrary such that $|\xi|=1$. As $n\ge 3$, there are $\eta,\mu\in \R^n$ such that 
\begin{equation}
\label{eq_600_8_vectors}
|\eta|=|\mu|=1, \quad \xi\cdot\eta=\xi\cdot\mu=\eta\cdot\mu=0.
\end{equation}
 Since $(\xi+i\eta,\xi\pm i\mu)\in \mathcal{A}$, it follows from \eqref{eq_600_8} that 
\[
\tilde T(\xi+i\eta, \xi\pm i\mu) =0\quad \text{in}\quad \Omega,
\]
where $\tilde T$ is the rank two tensor with coefficients $T^{jk}$, $j,k=1,\dots, n$. Therefore, by linearity, 
\[
\tilde T(\xi+i\eta, \xi) =0\quad \text{in}\quad \Omega.
\]
Changing $\eta$ to $-\eta$, by linearity, we get 
\begin{equation}
\label{eq_600_9}
\tilde T(\xi, \xi) =0\quad \text{in}\quad \Omega.
\end{equation}
Since $\xi\in \R^n$ is arbitrary vector, $|\xi|=1$,  by linearity and polarization of \eqref{eq_600_9}, we conclude that 
\begin{equation}
\label{eq_600_10}
\tilde T(\xi, \tilde \xi) =0\quad \text{in}\quad \Omega, \quad \xi,\tilde \xi\in \R^n.
\end{equation}
Hence, it follows from \eqref{int_iden_density_m=2_0}, \eqref{int_iden_density_m=2_0_T_0_j}, and \eqref{eq_600_10}
 that $T=0$ in $\Omega$. This completes the proof of Proposition~\ref{prop1} in the case $m= 2$.

\subsection{Proof of Proposition~\ref{prop1} in the case $m\geq 3$} Here we shall proceed by induction on $m$. To that end, we assume that Proposition~\ref{prop1} holds for $m-1$, and we shall prove that it holds for $m$. 

First letting $u_{m+1}=u_m=\dots=u_2=1$ in \eqref{int_iden_density}, we see that $T^{0\dots 0}=0$. Using that $T^{0\dots 0}=0$ and letting $u_{m+1}=u_m=\dots=u_3=1$ in \eqref{int_iden_density}, we show that the tensor of rank one with coefficients $T^{j0\dots 0}$, $j=1,\dots, n$, is equal to zero. Proceeding in the same way and finally letting $u_{m+1}=1$, by the induction hypothesis, we get that the tensor of rank $m-1$, whose coefficients are $T^{j_1\dots j_{m-1}0}$, $j_1,\dots, j_{m-1}=1,\dots, n$, is equal to zero.  In view of all of these, the integral identity  \eqref{int_iden_density} becomes
\begin{equation}
\label{int_iden_density_main_900}
\sum_{(l_1,\dots, l_{m+1})\in \pi(m+1)}\sum_{j_1,\dots, j_{m}=1}^n \int_\Omega
  T^{j_1\dots j_m}(x)
   \p_{x_{j_1}}u_{l_1}\dots \p_{x_{j_m}}u_{l_{m}}
   \nabla u_{l_{m+1}}\cdot \nabla u_{m+2}dx=0,
\end{equation}
for all  $u_l\in C^\infty(\overline{\Omega})$, $l=1,\dots, m+2$,  solving \eqref{lin_rand}. To show that the identity \eqref{int_iden_density_main_900} implies that the tensor of rank $m$ with coefficients $T^{j_1\dots j_m}$, $j_1,\dots, j_{m}=1,\dots, n$, vanishes in $\Omega$, we shall first prove the following result. 

\begin{lem}
\label{lem_1}
Assume that the integral identity \eqref{int_iden_density_main_900} holds for all $u_l\in C^\infty(\overline{\Omega})$, $l=1,\dots, m+2$,  solving \eqref{lin_rand}. Then 
\begin{equation}
\label{eq_900_6}
\sum_{j_1,\dots, j_{m}=1}^n T^{j_1\dots j_m}   \zeta_{j_1}\zeta_{j_2}\dots, \zeta_{j_{m-1}}\tilde \zeta_{j_m}=0\quad \text{in}\quad \Omega,
\end{equation}
for all $(\zeta,\tilde \zeta)\in \mathcal{A}$, $|\emph{\text{Re}}\,\zeta|=|\emph{\text{Re}}\,\tilde \zeta|=1$, where $\mathcal A$ is as in Definition~\ref{zeta_pair}.
\end{lem}

\begin{proof}
To prove \eqref{eq_900_6} we shall use suitable CGO solutions to \eqref{lin_rand} defined as in the proof of Proposition~\ref{prop1} in the case $m=2$. To that end, we let $p \in \Omega$,  $(\zeta,\widetilde \zeta) \in \mathcal A$, $|\text{Re}\,\zeta|=|\text{Re}\,\tilde \zeta|=1$,  $\sigma \in \R$, $\delta \in (0,1)$, and let  $\{\omega_j\}_{j=1}^{n-2}$ and $\{\widetilde{\omega}_j\}_{j=1}^{n-2}$ be  orthonormal sets that satisfy \eqref{eq_600_-10} and \eqref{eq_600_-10_1}, respectively.  Taking $\lambda>0$ sufficiently large, we set 
\begin{equation}
\label{eq_900_1}
u_1=u_2=\ldots=u_{m-1}=U_{\lambda\zeta}, \quad u_m= U_{-(m-1)\lambda\zeta},\quad u_{m+1}= U_{\lambda\tilde\zeta},\quad u_{m+2}=\widetilde U_{-\lambda\tilde\zeta},
\end{equation}
where $U_{\lambda\zeta}, U_{\lambda\tilde\zeta},  U_{-\lambda\tilde\zeta}\in C^\infty(\overline{\Omega})$ are given by \eqref{eq_600_1} and $U_{-(m-1)\lambda\zeta}\in C^\infty(\overline{\Omega})$ is given by 
\begin{equation}
\label{eq_900_2}
U_{-\lambda (m-1)\zeta}(x)= e^{-(m-1)\lambda\zeta \cdot x}\gamma_0(x)^{-\frac{1}{2}}\left( a(x)+r_{-\lambda(m-1)\zeta}(x)\right),
\end{equation}
where $a\in C^\infty(\R^n)$ is defined by \eqref{h_form}, and 
\begin{equation}
\label{eq_900_3}
\|r_{-\lambda(m-1)\zeta}\|_{C^1(\overline{\Omega})}=\mathcal{O}(\lambda^{-1}), 
\end{equation}
as $\lambda\to \infty$. The existence of such CGO solutions follows from Lemma \ref{lem_cgo}. 

Substituting \eqref{eq_900_1}, \eqref{eq_600_1}, \eqref{eq_900_2},  into the integral identity \eqref{int_iden_density_main_900}, multiplying by $\lambda^{-(m+2)}$, $\lambda\to \infty$, and using \eqref{eq_600_2}, \eqref{eq_600_3}, \eqref{eq_900_3},  we get  
\[
(\zeta\cdot\tilde \zeta)\int_\Omega \sum_{j_1,\dots, j_{m}=1}^n T^{j_1\dots j_m}\zeta_{j_1}\zeta_{j_2}\dots, \zeta_{j_{m-1}}\tilde \zeta_{j_m}\gamma_0^{-\frac{(m+2)}{2}}a^{m}\widetilde{a}^2dx=0,
\]
and therefore,
\begin{equation}
\label{eq_900_4}
\int_\Omega \sum_{j_1,\dots, j_{m}=1}^n T^{j_1\dots j_m}\zeta_{j_1}\zeta_{j_2}\dots, \zeta_{j_{m-1}}\tilde \zeta_{j_m}\gamma_0^{-\frac{(m+2)}{2}} Fdx=0.
\end{equation}
Here 
\[
F(x)=a^m(x) \widetilde{a}^2(x)=e^{i \sigma(m\zeta+2\tilde \zeta)\cdot (x-p)}\prod_{j=1}^{n-2}\chi\bigg(\frac{\omega_j\cdot (x-p)}{\delta}\bigg)^m \chi\bigg(\frac{\widetilde{\omega}_j\cdot (x-p)}{\delta}\bigg)^2, 
\]
$\delta\in (0,1)$ and $\sigma\in \R$. As $(\zeta,\tilde \zeta)\in \mathcal{A}$, we have $\supp(F)$ is contained in  a $\delta$--neighborhood of the ray \eqref{eq_600_6}, cf. the discussion after \eqref{eq_600_6}.

Multiplying \eqref{eq_900_4} by $\delta^{-(n-1)}$ and taking the limit as $\delta \to 0$, we deduce that
\begin{equation}
\label{eq_900_5}
\begin{aligned}
 \int_{\R} \sum_{j_1,\dots, j_{m}=1}^n &e^{i (m+2)\sigma t} \gamma_0^{-\frac{(m+2)}{2}}(p+t\text{Re}\,\zeta) \\
 &T^{j_1\dots j_m} (p+t\text{Re}\, \zeta)  \zeta_{j_1}\zeta_{j_2}\dots, \zeta_{j_{m-1}}\tilde \zeta_{j_m} \,dt=0,
\end{aligned}
\end{equation}
where we have extended $T^{j_1\dots j_m}$ and $\gamma_0$ to all of $\R^n$ by setting them to be zero outside $\Omega$. Now \eqref{eq_900_6} follows from \eqref{eq_900_5}, cf. the discussion before \eqref{eq_600_8}.
\end{proof}

\begin{rem}
When $m\ge 3$, one cannot conclude from \eqref{eq_900_6} directly that $T^{j_1\dots j_m}=0$ for all $j_1,\dots, j_{m}=1,\dots, n$. Indeed, taking $T^{j_1\dots j_m}=\delta_{j_1 j_2}$ for all $j_3,\dots, j_m=1,\dots, n$, we see that \eqref{eq_900_6} holds, as $\zeta\cdot\zeta=0$. 
\end{rem}

Hence, to show that $T^{j_1\dots j_m}=0$ for all $j_1,\dots, j_{m}=1,\dots, n$, we shall rely on the two lemmas below. 
\begin{lem}
\label{lem_first_two_indices}
Assume that the integral identity \eqref{int_iden_density_main_900} holds for all $u_l\in C^\infty(\overline{\Omega})$, $l=1,\dots, m+2$,  solving \eqref{lin_rand}. Then we have 
\begin{equation}
\label{eq_900_6_3}
\sum_{j_1,\ldots,j_m=1}^n T^{j_1\ldots j_m} \zeta_{j_1}\,\widetilde{\zeta}_{j_2}\, \p_{ x_{j_3}} u \ldots \p_{ x_{j_m}}u=0 \quad\text{in}\quad \Omega,
\end{equation}
for all  $(\zeta,\tilde \zeta)\in \mathcal{A}$, $|\emph{\text{Re}}\,\zeta|=|\emph{\text{Re}}\,\tilde \zeta|=1$, and all $u\in C^\infty(\overline{\Omega})$ solving \eqref{lin_rand}.
\end{lem}
\begin{proof}
We shall prove this lemma by induction. To that end, letting $(\zeta,\tilde \zeta)\in \mathcal{A}$, $|\text{Re}\,\zeta|=|\text{Re}\,\tilde \zeta|=1$, and letting $u\in C^\infty(\overline{\Omega})$ satisfy \eqref{lin_rand}, we assume that the following holds
\begin{equation}
\label{eq_900_13}
\sum_{j_1,\ldots,j_m=1}^n T^{j_1\ldots j_m}  
 \zeta_{j_1}\ldots \zeta_{j_{s}}\widetilde{\zeta}_{j_{s+1}}\p_{x_{j_{s+2}}}u\ldots \p_{x_{j_{m}}}u=0\quad \text{in}\quad \Omega,
\end{equation}
for $s=m-1,\dots, k+1$, with some  $1\le k\le m-2$. Note that \eqref{eq_900_13} with $s=m-1$ corresponds to  \eqref{eq_900_6} and is the basis for the induction. We shall prove that \eqref{eq_900_13} holds for $s=k$. In doing so we test the integral identity  \eqref{int_iden_density_main_900} with suitable choice of CGO solutions to the linear conductivity equation  \eqref{lin_rand} and use \eqref{eq_900_13}. Specifically,  taking $\lambda>0$ sufficiently large, we set
\begin{equation}
\label{eq_900_7}
u_1=u_2=\ldots=u_{k}=U_{\lambda\zeta}, \quad u_{k+1}= U_{-\lambda k\zeta},\quad u_{k+2}= U_{\lambda\tilde \zeta},\quad  u_{m+2}=U_{-\lambda\tilde \zeta},
\end{equation}
while 
\begin{equation}
\label{eq_900_8}
 u_{k+3}=\dots=u_{m+1}=u.
 \end{equation}
Here $U_{\lambda\zeta}, U_{\lambda\tilde\zeta},  U_{-\lambda\tilde\zeta}\in C^\infty(\overline{\Omega})$ are given by \eqref{eq_600_1} and $U_{-\lambda k\zeta}\in C^\infty(\overline{\Omega})$ is given by \eqref{eq_900_2} with $(m-1)$ being replaced by $k$.

Substituting \eqref{eq_900_7}, \eqref{eq_900_8}, \eqref{eq_600_1}, \eqref{eq_900_2},  into the integral identity \eqref{int_iden_density_main_900}, multiplying by $\lambda^{-(k+3)}$, letting $\lambda\to \infty$,  and using \eqref{eq_600_2}, \eqref{eq_600_3}, \eqref{eq_900_3},  we get  
\begin{equation}
\label{eq_900_9}
\begin{aligned}
&c_{m,k}\,(\zeta\cdot \widetilde\zeta) \underbrace{\int_{\Omega}\sum_{j_1,\ldots,j_m=1}^n \gamma_0^{-\frac{(k+3)}{2}}F\,T^{j_1\ldots j_m} \zeta_{j_1}\ldots \zeta_{j_{k}}\widetilde{\zeta}_{j_{k+1}}\p_{x_{j_{k+2}}}u \dots \p_{x_{j_m}}u\,dx}_{\textrm{I}}\\
&+d_{m,k}\underbrace{\int_{\Omega}\sum_{j_1,\ldots,j_m=1}^n \gamma_0^{-\frac{(k+3)}{2}}F\,T^{j_1\ldots j_m} \zeta_{j_1}\ldots \zeta_{j_{k+1}}\widetilde{\zeta}_{j_{k+2}} \p_{x_{j_{k+3}}}u \dots \p_{x_{j_m}}u \nabla u\cdot \widetilde \zeta\,dx}_{\textrm{II}}=0
\end{aligned}
\end{equation}
with some non-zero constants $c_{m,k}$ and $d_{m,k}$ that only depend on $m$ and $k$. Here, the function $F$ is given by 
\begin{align*}
F(x)&=a^{k+1}(x) \widetilde{a}^2(x)\\
&=e^{i \sigma((k+1)\zeta+2\tilde \zeta)\cdot (x-p)}\prod_{j=1}^{n-2}\chi\bigg(\frac{\omega_j\cdot (x-p)}{\delta}\bigg)^{k+1} \chi\bigg(\frac{\widetilde{\omega}_j\cdot (x-p)}{\delta}\bigg)^2, 
\end{align*}
$\delta\in (0,1)$ and $\sigma\in \R$. It follows from \eqref{eq_900_13} with $s=k+1$ that II in \eqref{eq_900_9} vanishes and we conclude that $I$ must also vanish, i.e. 
\begin{equation}
\label{eq_900_10}
\int_{\Omega}\sum_{j_1,\ldots,j_m=1}^n \gamma_0^{-\frac{(k+3)}{2}}F\,T^{j_1\ldots j_m} \zeta_{j_1}\ldots \zeta_{j_{k}}\widetilde{\zeta}_{j_{k+1}}\p_{x_{j_{k+2}}}u \dots \p_{x_{j_m}}u\,dx=0.
\end{equation}

Multiplying \eqref{eq_900_10} by $\delta^{-(n-1)}$ and taking the limit as $\delta \to 0$, we observe that
\begin{equation}
\label{eq_900_11}
\begin{aligned}
 \int_{\R} \sum_{j_1,\ldots,j_m=1}^n &e^{i\sigma t (k+3)}\gamma_0^{-\frac{(k+3)}{2}}(p+t\Re\zeta)T^{j_1\ldots j_m}  (p+t\Re\zeta)\\
&\zeta_{j_1}\ldots \zeta_{j_{k}}\widetilde{\zeta}_{j_{k+1}} \p_{x_{j_{k+2}}}u(p+t\Re\zeta) \dots
 \p_{x_{j_m}}u (p+t\Re\zeta) \,dt=0,
\end{aligned}
\end{equation}
where we have extended $T^{j_1\dots j_m}$ and $\gamma_0$ to all of $\R^n$ by setting them to be zero outside $\Omega$. As above, cf. the discussion before \eqref{eq_600_8}, we conclude from \eqref{eq_900_11} that 
\begin{equation}
\label{eq_900_12}
\sum_{j_1,\ldots,j_m=1}^n T^{j_1\ldots j_m}  
 \zeta_{j_1}\ldots \zeta_{j_{k}}\widetilde{\zeta}_{j_{k+1}} \p_{x_{j_{k+2}}}u \dots
 \p_{x_{j_m}}u =0\quad \text{in}\quad \Omega. 
\end{equation}
This shows that \eqref{eq_900_13} holds for $s=k$.  The proof of Lemma \ref{lem_first_two_indices} is completed by setting $s=1$ in  \eqref{eq_900_13}.
\end{proof}

\begin{lem}
\label{lem_full}
Assume that \eqref{eq_900_6_3} holds for  all  $(\zeta,\tilde \zeta)\in \mathcal{A}$, $|\emph{\text{Re}}\,\zeta|=|\emph{\text{Re}}\,\tilde \zeta|=1$, and all $u\in C^\infty(\overline{\Omega})$ solving \eqref{lin_rand}. Then the tensor $\tilde T=0$ in $\Omega$, where $\tilde T$ is the tensor of rank $m$ with coefficients $T^{j_1 \dots j_m}$, $j_1,\dots, j_m=1,\dots, n$.  
\end{lem}
\begin{proof}
First, arguing as after \eqref{eq_600_8}, we conclude from \eqref{eq_900_6_3}  that 
\begin{equation}
\label{eq_900_17}
 \sum_{j_1,j_2,j_3,\ldots,j_m=1}^n T^{j_1j_2 j_{3}\ldots j_{m}} \xi^{(1)}_{j_1}\xi^{(2)}_{j_2} \p_{ x_{j_3}} u \ldots \p_{ x_{j_m}}u=0 \quad\text{in}\quad \Omega,
\end{equation}
for all $\xi^{(1)},\xi^{(2)}\in \R^n$. Next  via polarization of \eqref{eq_900_17}, see \cite{Th14},  we obtain that
\begin{equation}
\label{eq_900_18}
 \sum_{j_1,j_2,j_3,\ldots,j_m=1}^n T^{j_1j_2 j_{3}\ldots j_{m}} \xi^{(1)}_{j_1}\xi^{(2)}_{j_2} \p_{ x_{j_3}} u_3 \ldots \p_{ x_{j_m}}u_m=0 \quad\text{in}\quad \Omega,
\end{equation}
for all $u_3,\dots, u_m \in C^\infty(\overline{\Omega})$ solving \eqref{lin_rand}. .

Letting $\lambda>0$ sufficiently large,  $\zeta^{(j)}\in \C^n$ be  such that $\zeta^{(j)}\cdot\zeta^{(j)} =0$, $|\text{Re}\,\zeta^{(j)}|=|\text{Im}\,\zeta^{(j)}|=1$, $j=3,\dots, m$, we set 
\begin{equation}
\label{eq_900_19}
u_{j}=U_{\lambda\zeta^{(j)}}=e^{\lambda\zeta^{(j)}\cdot x}\gamma_0^{-1/2}(1+r_{\lambda\zeta^{(j)}})\in C^\infty(\overline{\Omega}), \quad j=3,\dots, m,
\end{equation}
with $\|r_{\lambda\zeta^{(j)}}\|_{C^1(\overline{\Omega})}=\mathcal{O}(\lambda^{-1})$, as $\lambda\to \infty$. 

Substituting \eqref{eq_900_19} into \eqref{eq_900_18}, and multiplying by $\lambda^{-(m-2)}$, we see that 
\begin{equation}
\label{eq_900_20}
 \sum_{j_1,j_2,j_3,\ldots,j_m=1}^n T^{j_1j_2 j_{3}\ldots j_{m}} \xi^{(1)}_{j_1}\xi^{(2)}_{j_2}\zeta^{(3)}_{j_3} \ldots \zeta^{(m)}_{j_m}=0 \quad\text{in}\quad \Omega.
\end{equation}
Let $\xi \in \R^n$ be arbitrary such that $|\xi|=1$, and choose $\eta \in \R^n$ such that $|\eta|=1$ and $\xi\cdot \eta=0$. Letting $\zeta=\xi+i\eta$, using that $\zeta\cdot\zeta=\overline{\zeta}\cdot\overline{\zeta}=0$ and linearity, we get from \eqref{eq_900_20} that 
\begin{equation}
\label{eq_900_21}
\tilde T(\xi^{(1)},\xi^{(2)}, 2\xi,\dots,2\xi)= \tilde T(\xi^{(1)},\xi^{(2)}, \zeta+\overline{\zeta},\dots, \zeta+\overline{\zeta})=0 \quad\text{in}\quad \Omega.
\end{equation}
By linearity and polarization of \eqref{eq_900_21}, we conclude that $\tilde T=0$ in $\Omega$. 
\end{proof}

Lemma \ref{lem_full} completes the proof of Proposition~\ref{prop1} in the case $m\ge 3$.

\begin{appendix}

\section{Proof of Lemma \ref{lem_cgo}}
\label{proof_CGO}

While Lemma \ref{lem_cgo} is known, see  \cite{BU10}, \cite[Proposition 2]{FO20},  \cite{LLLSa20}, \cite{KU20}, we shall present here a very simple proof of it  for the convenience of the reader.   In doing so we shall use the approach of \cite{H96} which is based on Fourier series, see also \cite{Sa08}, extending it to get good remainder estimates in an arbitrary Sobolev space $H^m(\Omega)$.  First we have, see \cite{SU87}, 
\begin{equation}
\label{eq_500_1_0}
-\gamma_0^{-1/2}\circ L_{\gamma_0}\circ \gamma_0^{-1/2}=-\Delta+q, \quad q=\frac{\Delta\gamma_0^{1/2}}{\gamma_0^{1/2}}\in C^\infty(\overline{\Omega}),
\end{equation}
where the conductivity operator $L_{\gamma_0}$ is defined as follows $L_{\gamma_0}:=\nabla\cdot (\gamma_0\nabla \cdot)$.  We would like to construct CGO solutions to the Schr\"odinger equation
\begin{equation}
\label{eq_500_1}
(-\Delta+q)u=0\quad \text{in}\quad \Omega,
\end{equation}
of the form 
\begin{equation}
\label{eq_500_2}
u_{\lambda\zeta}(x)=e^{\lambda\zeta\cdot x}(a(x)+r_{\lambda\zeta}(x)),
\end{equation}
where $\lambda>0$ is a large parameter, $ 0\ne \zeta\in \C^n$ is independent of $\lambda$ such that $\zeta\cdot\zeta=0$, $a$ is a smooth amplitude, and $r$ is the remainder term. Then it follows from \eqref{eq_500_1_0} that 
\[
U_{\lambda\zeta}(x)=e^{\lambda\zeta\cdot x}\gamma_0^{-1/2}(a(x)+r_{\lambda\zeta}(x))
\]
are CGO solutions to \eqref{linear_eq}.  Substituting \eqref{eq_500_2} into \eqref{eq_500_1}, we get 
\[
e^{-\lambda\zeta\cdot x}(-\Delta+q)e^{\lambda\zeta\cdot x}(a(x)+r_{\lambda\zeta}(x))=0  \quad \text{in}\quad \Omega,
\]
and therefore, setting $r=r_{\lambda\zeta}$, we have
\begin{equation}
\label{eq_500_3}
(-\Delta-2\lambda\zeta\cdot\nabla +q)r=-(-\Delta-2\lambda\zeta\cdot\nabla +q)a \quad \text{in}\quad \Omega.
\end{equation}

To solve \eqref{eq_500_3} we assume for simplicity that $\Omega \subset Q:=[-\pi,\pi]^n$. Note that everything works without this extra assumption if we replace $\Omega$ by its image under the map $\R^n\ni x\mapsto\kappa x\in \R^n$ for some sufficiently small fixed $\kappa>0$.  First we shall solve 
\begin{equation}
\label{eq_500_3_new}
(-\Delta-2\lambda\zeta\cdot\nabla )r=f \quad \text{in}\quad Q,
\end{equation}
where $f\in L^2(Q)$.  Writing $\zeta=\omega_1+i\omega_2$, $\omega_1,\omega_2\in \R^n$, we see that $\omega_1\cdot\omega_2=0$ and $|\omega_1|=|\omega_2|=:\alpha$. We may assume without loss of generality that $\omega_1=\alpha e_1$ and $\omega_2=\alpha e_2$, where $e_1$ and $e_2$ are the first two vectors in the standard basis of $\R^n$.  Thus, \eqref{eq_500_3_new} becomes
\begin{equation}
\label{eq_500_3_new_1}
(-\Delta-2\lambda \alpha \p_{x_1}-2i\lambda\alpha \p_{x_2} )r=f \quad \text{in}\quad Q.
\end{equation}
Letting $v_l(x)=e^{i(l+\frac{1}{2}e_1)\cdot x}$, $l\in \Z^n$, and noting that $(v_l)$ forms an orthonormal basis in $L^2(Q,dx/(2\pi)^n)$, see \cite{Sa08}, we have
\[
f=\sum_{l\in \Z^n}f_lv_l,
\]
where $f_l=(f,v_l)_{L^2(Q)}=(2\pi)^{-n}\int_{Q}f\overline{v_l}dx$, $\|f\|_{L^2(Q)}^2=\sum_{l\in\Z^n}|f_l|^2$. Looking for a solution $r$ of \eqref{eq_500_3_new_1} in the form $r=\sum_{l\in \Z^n}r_lv_l$, we are led to the following equation, 
\[
p_l r_l=f_l, \quad p_l=\bigg(l+\frac{1}{2}e_1\bigg)^2-i2\lambda\alpha \bigg(l_1+\frac{1}{2}\bigg)+2\lambda\alpha l_2, \quad l\in \Z^n.
\]
Using that $|\text{Im}p_l|\ge \lambda\alpha $ and letting $r_l:=f_l/p_l$, we get $|r_l|\le |f_l|/(\lambda \alpha)$. Thus, $\|r\|_{L^2(Q)}\le \frac{1}{\lambda \alpha}\|f\|_{L^2(Q)}$.

Now if $f\in H^m(Q)$, $m\ge 0$, where $H^m(Q)$ is the Sobolev space, equipped with the norm
\[
\|f\|_{H^m(Q)}^2=\|(1-\Delta)^{m/2}f\|_{L^2(Q)}^2=\sum_{l\in \Z^n}\bigg(1+\bigg|l+\frac{1}{2}e_1\bigg|^2\bigg)^m|f_l|^2,
\] 
we see that $\|r\|_{H^m(Q)}\le \frac{1}{\lambda\alpha}\|f\|_{H^m(Q)}$. 

Now letting $f\in H^m(\Omega)$ and extending it continuously to  $H^m(Q)$, we have constructed the solution 
$r\in H^m(\Omega)$ to  the equation
\begin{equation}
\label{eq_500__3_new_2}
(-\Delta-2\lambda\zeta\cdot\nabla )r=f \quad \text{in}\quad \Omega
\end{equation}
satisfying 
\begin{equation}
\label{eq_500__3_new_3}
\|r\|_{H^m(\Omega)}\le \frac{C}{\lambda}\|f\|_{H^m(\Omega)}. 
\end{equation}
We denote by $G_{\lambda\zeta}$ the solution operator
\[
G_{\lambda\zeta}:H^m(\Omega)\to H^m(\Omega), \quad f\mapsto r,
\]
where $r$ is the solution to \eqref{eq_500__3_new_2}   that we have just constructed. It follows from \eqref{eq_500__3_new_3} that 
\[
\|G_{\lambda\zeta}\|_{H^m(\Omega)\to H^m(\Omega)}=\mathcal{O}(\lambda^{-1}), \quad \lambda\to \infty.
\] 

To solve \eqref{eq_500_3},  first let $a\in C^\infty(\overline{\Omega})$ be any solution of the transport equation
\begin{equation}
\label{eq_500_4}
\zeta\cdot\nabla a=0 \quad \text{in}\quad \Omega.
\end{equation}
Then it follows from \eqref{eq_500_3} and \eqref{eq_500_4} that we would like to find $r$ such that 
\begin{equation}
\label{eq_500_5}
(-\Delta-2\lambda\zeta\cdot\nabla +q)r=-(-\Delta +q)a \quad \text{in}\quad \Omega.
\end{equation}
Note that $(-\Delta +q)a\in H^m(\Omega)$. 
Looking for a solution $r$ of \eqref{eq_500_5} in the form $r=G_{\lambda\zeta} \tilde r$, we get that $\tilde r$ should solve the equation 
\begin{equation}
\label{eq_500_8}
(I +qG_{\lambda\zeta})\tilde r=-(-\Delta +q)a \quad \text{in}\quad \Omega.
\end{equation}
As $q\in C^\infty(\overline{\Omega})$, we have $\|qG_{\lambda\zeta}\|_{H^m(\Omega)\to H^m(\Omega)}=\mathcal{O}(\lambda^{-1})$, as $\lambda\to \infty$. Then by Neumann series, for $\lambda>0$ sufficiently large, we see that \eqref{eq_500_8} has a solution $\tilde r\in H^m(\Omega)$ such that $\|\tilde r\|_{H^m(\Omega)}=\mathcal{O}(1)\|(-\Delta +q)a \|_{H^m(\Omega)}$. Therefore, $\|r\|_{H^m(\Omega)}=\mathcal{O}(\lambda^{-1})\|(-\Delta +q)a \|_{H^m(\Omega)}$, as $\lambda\to \infty$. 

Performing the above construction on a bounded open set $\tilde \Omega$ with $C^\infty$ boundary such that $\Omega\subset \subset\tilde \Omega$,  using elliptic regularity and the Sobolev embedding $H^{m}(\Omega)\subset C^1(\overline{\Omega})$, $m>n/2+1$, we complete the proof of  Lemma \ref{lem_cgo}.

\section{Well-posedness of the Dirichlet problem for a quasilinear conductivity equation}

\label{App_well_posedness}

The purpose of this appendix is to show the well-posedness of the Dirichlet problem for a quasilinear conductivity equation without analyticity assumptions. The argument is standard and is given here for completeness and convenience of the reader,  see \cite[Proposition 2.1]{LLLSa20} for similar arguments in the case of semilinear elliptic equations. 

Let $\Omega\subset \R^n$, $n\ge 2$, be a bounded open set with $C^\infty$ boundary. Let $k\in \N\cup \{0\}$ and $0<\alpha<1$ and let $C^{k,\alpha}(\overline{\Omega})$ be the standard H\"older space  on $\Omega$, see  \cite{Hormander_1976}.  We write $C^{\alpha}(\overline{\Omega})=C^{0,\alpha}(\overline{\Omega})$. 

Let $\rho\in \R$ and consider the Dirichlet problem for the following isotropic quasilinear conductivity equation, 
\begin{equation}
\label{eq_app_ref_1}
\begin{cases} 
\nabla\cdot  (\gamma(x,u, \nabla u)\nabla u)=0& \text{in}\quad \Omega,\\
u=\rho+f & \text{on}\quad \p \Omega.
\end{cases}
\end{equation}
We assume that the function $\gamma: \overline{\Omega} \times \R \times \R^n\to \R$ satisfies the following conditions:
\begin{enumerate}
\item[(A1)]{$0<\gamma(\cdot,\rho,0) \in C^{\infty}(\overline{\Omega})$, for $\rho\in \R$,} 
\item[(A2)] {the map $\R\times \R^n \ni (\rho,\mu) \to\gamma(\cdot,\rho,\mu)$ is $C^\infty$ with values in the H\"{o}lder space $C^{1,\alpha}(\overline{\Omega};\R)$ for some $\alpha \in (0,1)$.}  
\end{enumerate} 

We have the following result.

\begin{thm}
\label{thm_well-posedness}
Let $\rho\in \R$ be fixed. Then  under the above assumptions, there exist $\delta>0$, $C>0$ such that for any $f\in  B_{\delta}(\p \Omega;\R):=\{f\in C^{2,\alpha}(\p \Omega;\R): \|f\|_{C^{2,\alpha}(\p \Omega;\R)}< \delta\}$, the problem \eqref{eq_app_ref_1} has a solution $u=u_{\lambda, f}\in C^{2,\alpha}(\overline{\Omega};\R)$ which satisfies
\[
\|u-\rho \|_{C^{2,\alpha}(\overline{\Omega};\R)}\le C\|f\|_{C^{2,\alpha}(\p \Omega;\R)}.
\]
The solution $u$ is unique within the class $\{u\in C^{2,\alpha}(\overline{\Omega};\R): \|u-\rho\|_{C^{2,\alpha}(\overline{\Omega};\R)}< C\delta \}$ and the map 
\[
B_\delta(\p \Omega;\R)\to C^{2,\alpha}(\overline{\Omega};\R), \quad f\mapsto  u,
\]
is $C^\infty$. Furthermore, the map
\begin{equation}
\label{eq_app_ref_2_normal}
B_\delta(\p \Omega;\R)\to C^{1,\alpha}(\overline{\Omega};\R), \quad f\mapsto \p_\nu u|_{\p \Omega}
\end{equation}
is also $C^\infty$.
\end{thm}

\begin{proof}
Following \cite[Proposition 2.1]{LLLSa20}, we shall make use of the implicit function theorem for $C^\infty$ maps between real Banach spaces, see \cite[Theorem 10.6]{RR06}. In doing so, we let 
\[
B_1=C^{2,\alpha}(\p \Omega;\R), \quad  B_2=C^{2,\alpha}(\overline{\Omega};\R), \quad B_3=C^{\alpha}(\overline{\Omega};\R)\times C^{2,\alpha}(\p \Omega;\R). 
\]
Consider the map,
\begin{equation}
\label{eq_app_ref_2}
F:B_1\times B_2\to B_3, \quad F(f,u)=(\nabla\cdot (\gamma(x,u,\nabla u)\nabla u), u|_{\p \Omega}-\rho-f).
\end{equation}
First we claim that  $F$ has the mapping property \eqref{eq_app_ref_2}.  Indeed, as $C^{1,\alpha}(\overline{\Omega};\R)$ is an  algebra under pointwise multiplication, see \cite[Theorem A.7]{Hormander_1976}, we only need to see that $\gamma(x,u,\nabla u)\in C^{1,\alpha}(\overline{\Omega};\R)$. This follows from the fact that if $v\in C^\infty(\R)$ and $w\in C^{1,\alpha}(\overline{\Omega};\R)$ then the composition $v\circ w\in C^{1,\alpha}(\overline{\Omega};\R)$, see  \cite[Theorem A.8]{Hormander_1976}. 

Let us check that the map $F$ in \eqref{eq_app_ref_2} is $C^\infty$. To that end, it suffices to check that the map 
\begin{equation}
\label{eq_app_ref_3}
C^{2,\alpha}(\overline{\Omega};\R)\ni u\mapsto \gamma(x,u,\nabla u)\in C^{1,\alpha}(\overline{\Omega};\R)
\end{equation}
is $C^\infty$. In doing so, letting $\lambda=(\rho,\mu)\in \R\times\R^n$, we Taylor expand $\gamma(x,\cdot)$ at $\lambda_0$, 
\begin{equation}
\label{eq_app_ref_4}
\gamma(x,\lambda_0+\lambda)=\sum_{|\beta|\le N}\frac{(\p_\lambda^\beta \gamma)(x,\lambda_0)}{\beta!} \lambda^\beta+R_{\lambda_0,\lambda},
\end{equation}
where the remainder $R_{\lambda_0,\lambda}$ is given by 
\begin{equation}
\label{eq_app_ref_5}
R_{\lambda_0,\lambda}:=
(N+1)\sum_{|\beta|=N+1} \frac{\lambda^\beta}{\beta!}\int_0^1 (1-t)^N(\p_\lambda^\beta\gamma)(x,\lambda_0+t\lambda)dt,
\end{equation}
$N\ge 0$.
Therefore,  to prove that the map \eqref{eq_app_ref_3} is $C^\infty$, letting $\lambda_0=(u(x), \nabla u(x))$, $u \in C^{2,\alpha}(\overline{\Omega};\R)$, be fixed,  and letting $\lambda=(h(x), \nabla h(x))$, $h\in C^{2,\alpha}(\overline{\Omega};\R)$, in \eqref{eq_app_ref_4} and \eqref{eq_app_ref_5}, we have to check that the map 
\begin{equation}
\label{eq_app_ref_6}
C^{2,\alpha}(\overline{\Omega};\R)\ni u \mapsto (\p_\lambda^\beta \gamma)(x,u, \nabla u)\in C^{1,\alpha}(\overline{\Omega};\R)
\end{equation}
is continuous for all $|\beta|\ge 0$, and 
\begin{equation}
\label{eq_app_ref_7}
R_{(u,\nabla u), (h,\nabla h)}=o((h,\nabla h)^N)\quad \text{in}\quad  C^{1,\alpha}(\overline{\Omega};\R),
\end{equation}
as $h\to 0$ in $C^{1,\alpha}(\overline{\Omega};\R)$. 
 The continuity of the map \eqref{eq_app_ref_6}  follows from the fact that if $v\in C^\infty(\R)$ then the map $C^{1,\alpha}(\overline{\Omega};\R)\ni w\mapsto v\circ w\in C^{1,\alpha}(\overline{\Omega};\R)$ is continuous, see  \cite[Theorem A.8]{Hormander_1976}. Now for $\|(h,\nabla h)\|_{C^{1,\alpha}(\overline{\Omega};\R)}\le 1$, we have $\|(\p_\lambda^\beta \gamma)(x,u(x)+t h(x), \nabla u(x)+t\nabla h(x))\|_{C^{1,\alpha}(\overline{\Omega};\R)}\le C(u)$, uniformly in $t\in (0,1)$, where  $C(u)>0$ is a constant which depends on $u$.  Therefore, 
 \[
 \|R_{(u,\nabla u), (h,\nabla h)}\|_{C^{1,\alpha}(\overline{\Omega};\R)}\le C\|(h,\nabla h)\|_{C^{1,\alpha}(\overline{\Omega};\R)}^{N+1}, 
 \]
showing \eqref{eq_app_ref_7}. 

Note that $F(0,\rho)=0$ and the partial differential $\p_uF(0,\rho): B_2\to B_3$ is given by 
\[
\p_uF(0,\rho) v=(\nabla\cdot (\gamma(x,\rho, 0)\nabla v), v|_{\p \Omega}).
\]
Writing the equation $\nabla\cdot (\gamma(x,\rho, 0)\nabla v)=0$ as $\Delta v+\nabla(\log \gamma(x,\rho,0)) \cdot \nabla v=0$, and using (A1),  we see from \cite[Theorem 6.15]{Gil_Tru_book} that the map  $\p_u F(0,\rho ): B_2\to B_3$ is a linear isomorphism.

An application of  the implicit function theorem, see \cite[Theorem 10.6 and Remark 10.5]{RR06}, shows that there exists  $\delta>0$ and a unique $C^\infty$ map $S: B_\delta(\p \Omega;\R)\to C^{2,\alpha}(\overline{\Omega};\R)$ such that $S(0)=\rho$ and $F(f, S(f))=0$ for all $f\in B_\delta(\p \Omega;\R)$. Letting $u=S(f)$ and using that $S$ is Lipschitz continuous and $S(0)=\rho$, we have
\[
\|u-\rho\|_{C^{2,\alpha}(\overline{\Omega};\R)}\le C\|f\|_{C^{2,\alpha}(\p \Omega;\R)}.
\]
Since the operation of taking the normal derivative and restricting it to the boundary is a linear map $C^{2,\alpha}(\overline{\Omega};\R)\to C^{1,\alpha}(\p \Omega;\R)$, 
\eqref{eq_app_ref_2_normal} follows. 
\end{proof}

\end{appendix}

\section*{Acknowledgements}

C.C. was supported by NSF of China under grant 11931011. A.F. gratefully acknowledges support from the Fields institute for research in mathematical sciences.  The work of Y.K. is partially supported by  the French National Research Agency ANR (project MultiOnde) grant ANR-17-CE40-0029. The research of K.K. is partially supported by the National Science Foundation (DMS 1815922). The research of G.U. is partially supported by NSF, a Walker Professorship at UW and a Si-Yuan Professorship at IAS, HKUST.

\end{document}